\documentclass[11pt]{amsart}
\usepackage[latin1]{inputenc} % ±ØÐë¼ÓÕâ¸ö£¬·ñÔòÍ¶¸å±àÒë²»³É¹¦
\usepackage{listings, tcolorbox}% ±ØÐë¼ÓÕâ¸ö£¬·ñÔòÍ¶¸å±àÒë²»³É¹¦

\usepackage{amssymb, amsmath, amsthm}
\usepackage[colorlinks=true,linkcolor=blue,citecolor=red]{hyperref}
\usepackage{color}
\usepackage{graphics}
\usepackage{subfigure}
\usepackage{graphicx}  %²åÈëÍ¼Æ¬µÄºê°ü
\usepackage{float}  %ÉèÖÃÍ¼Æ¬¸¡¶¯Î»ÖÃµÄºê°ü
\usepackage{epsfig}
\usepackage{amssymb}
\usepackage{epstopdf}
\usepackage{enumerate}
\usepackage{appendix}
\usepackage{amscd}
\usepackage{longtable,booktabs}
\usepackage{mathrsfs}
\usepackage{tikz}
\usepackage{bm}
\usepackage[T1]{fontenc}

\numberwithin{equation}{section}
\newtheorem{theorem}{Theorem}[section]

\newtheorem{lemma}[theorem]{Lemma}
\newtheorem{prop}[theorem]{Proposition}

\newtheorem{coro}[theorem]{Corollary}
\theoremstyle{definition}

\newtheorem{re}[theorem]{Remark}

\def\spn{\mathop{\mathrm{span}}}

\newcommand{\ii}{\mathrm i}
\newcommand{\one}{\mathbf 1}
\newcommand{\Ima}{\operatorname{Im}}
\newcommand{\Rea}{\operatorname{Re}}
\newcommand{\Tr}{\operatorname{Tr}}
\newcommand{\diag}{\operatorname{diag}}
\newcommand{\Hd}{\mathcal H_{\rm d}}
\newcommand{\Hc}{\mathcal H_{\rm c}}
\newcommand{\Hp}{L^2_+(\R)}
\newcommand{\Ran}{\operatorname{Ran}}
\newcommand{\Dom}{\operatorname{Dom}}
\newcommand{\Ker}{\operatorname{Ker}}
\newcommand{\Sch}{\mathfrak S}
\newcommand{\eps}{\varepsilon}

\newcommand{\usol}{u_{\rm sol}}

\newenvironment{Acknowledgment}
{\begin{trivlist}\item[]\textbf{Acknowledgments }}{\end{trivlist}}

\newcommand{\R}{\mathbb{R}}
\newcommand{\C}{\mathbb{C}}
\newcommand{\N}{\mathbb{N}}
\newcommand{\Z}{\mathbb{Z}}


\usepackage[shortlabels]{enumitem}
\usepackage{cleveref}
\crefname{prop}{proposition}{propositions}
\Crefname{prop}{Proposition}{Propositions}
\crefname{coro}{corollary}{corollaries}
\Crefname{coro}{Corollary}{Corollaries}

\date{}

\begin{document}
\title[Soliton resolution conjecture and soliton gases  for the Benjamin-Ono equation]
{Soliton resolution conjecture and soliton gases for the Benjamin-Ono equation with infinite-order multisolitons in the sharp well-posedness space}
%\today
\author[Pan]{Hong-Yu Pan}
\address{\textnormal{Hong-Yu Pan} \newline \indent School of Mathematics, China University of Mining and Technology, Xuzhou 221116, China}
\email{hypan@cumt.edu.cn}
\author[Tian]{Shou-Fu Tian$^{*,\dag}$}
\address{\textnormal{Shou-Fu Tian} (Corresponding author) \newline \indent School of Mathematics, China University of Mining and Technology, Xuzhou 221116, China}
\email{sftian@cumt.edu.cn}
\thanks{$^{*}$Corresponding authors(sftian@cumt.edu.cn, shoufu2006@126.com (S.F. Tian)).\\
	\hspace*{3ex}$^\dag$This author is contributed equally as the first author.}
\begin{abstract}
We prove the soliton resolution conjecture for infinite-order multisoliton solutions in the sharp well-posedness
space $H^s(\mathbb R)$, $s>-\frac12$, and construct soliton gases
in $H^{s}(\mathbb R)$ with arbitrary $s\in\mathbb R$, for the Benjamin--Ono equation. %These results extend infinite-order soliton resolution in the sharp well-posedness space, whereas previous results were established in $H^{\frac12}(\mathbb R)$, and extend finite multisoliton constructions to soliton gases. 
In particular, our work addresses two important open problems about inverse spectral for Lax operator, including the problem of spectral characterization of infinite-order multisolitons and that of determining the range of their asymptotic parameters.

A central ingredient of our work is a spectral decomposition
adapted to rough initial data. This extends the global soliton--radiation description to weakly
weighted $H^s$ data while allowing infinite negative eigenvalues.
Another important ingredient is a continuous norming action that
extends finite inverse spectral shifts to gas limits
while retaining a full continuous norming weight.
Our framework provides a basis for studying infinite-order soliton
resolution and the construction and asymptotics of soliton gases
for nonlocal equations.
\end{abstract}
\maketitle
\tableofcontents

\section{Introduction}
\label{sec:introduction}
In this work, we consider the Cauchy problem for the Benjamin--Ono
(BO) equation
\begin{equation}\label{eq:BO-introduction}
 \left\{
 \begin{array}{l}
  \partial_tu-\partial_x|D_x|u+\partial_x(u^2)=0,\\
  u(0,x)=u_0(x),
 \end{array}
 \right.
 \enspace (t,x)\in\R^2,
\end{equation}
where $u$ is real-valued and the Fourier multiplier $|D_x|$ is given by
$$
 |D_x|u(x)=\frac1\pi\operatorname{p.v.}
 \int_\R\frac{u'(y)}{x-y}\,dy,
$$
with
\begin{equation*}
 \operatorname{p.v.}\int_\R\frac{f(y)}{x-y}\,dy
 :=\lim_{\varepsilon\to0^+}
 \int_{|x-y|>\varepsilon}\frac{f(y)}{x-y}\,dy.
\end{equation*}
The BO equation describes the weakly nonlinear propagation of
internal waves in a stratified fluid of great depth
\cite{Benjamin,Ono75}. The unknown represents the displacement of
the interface from equilibrium. For the derivation and mathematical
theory of the equation, we refer to Klein and Saut \cite{KS21}.

The study of the Cauchy problem has progressed from the classical
Sobolev theory of Saut \cite{Saut} to the gauge transformation of
Tao \cite{TaoBO} and the global $L^2(\R)$ theory of Ionescu and
Kenig \cite{IonescuKenig}. Further developments include the $L^2$
analysis of Molinet and Pilod \cite{MP12} and the small-data
dispersive estimates of Ifrim and Tataru \cite{IfrimTataru}.
On the torus, G\'erard, Kappeler and Topalov \cite{GKT} constructed
nonlinear Fourier coordinates and proved sharp global
well-posedness for $s>-\frac12$. Killip, Laurens and Vi\c{s}an
\cite{KLV24} subsequently established global well-posedness on both
the line and the torus throughout this Sobolev range.

Once the global evolution is established, a natural problem is to
determine its long-time behavior. The KdV computations of Zabusky
and Kruskal \cite{ZabuskyKruskal} revealed the emergence and
interaction of solitary waves. Inverse scattering and the
nonlinear steepest descent method of Deift and Zhou \cite{DZ93}
provide a route to this description for classical integrable
equations; see \cite{BJM18} for the focusing nonlinear
Schr\"odinger equation.

For BO, G\'erard's explicit formula \cite{Gerard}, extended to
real $L^2(\R)$ data by Chen \cite{Chen}, expresses the solution
through the initial Lax and coordinate operators. Finite
multisoliton formulas \cite{Matsuno79}, inverse spectral
coordinates \cite{Sun21} and determinant formulas for rational
data \cite{BGGM} provide finite-dimensional versions of this
description. Gassot, G\'erard and Miller \cite{GGM} proved global
resolution into finitely many solitons and free radiation for
localized data. Related developments include asymptotic stability
of a soliton \cite{KenigMartel}, orbital stability of multisolitons
\cite{BKV26}, and Wu's simplicity and direct scattering results
\cite{Wu,WuDirect}. For arbitrary $L^2$ data, Badreddine, Killip
and Vi\c{s}an \cite[Theorem~1.5]{BKVasymp} proved strong $L^2$
convergence to a soliton or to zero on every fixed-width window
moving at a prescribed velocity.

Recently, Gassot and G\'erard \cite{GG} studied a class of
infinite-order multisolitons given initially by $L^2$ sums of
positive soliton profiles. These data belong to every $H^r$,
and their negative Lax eigenvalues are summable in absolute
value and accumulate at zero. Their solutions resolve without
radiation in an ordered $L^2$ limit, while the complete soliton
series gives convergence in $L^\infty$ and $\dot H^r$ for
$r\ge\frac12$. The index $\frac12$ therefore specifies a
convergence norm in that result; the prescribed initial data
are smooth.

The passage from finitely many to infinitely many solitons
allows waves of arbitrarily small amplitude and velocity.
Such configurations require control of the accumulated
contribution of small solitons, which a fixed finite
approximation does not capture. Spectral estimates must also
remain effective as the gaps between negative eigenvalues
shrink. For general integrable initial data, the continuous
spectral contribution produces a further term in the long-time
description. The infinite positive pole sums of \cite{GG}
have infinite $L^1$ mass, whereas integrable data can have
infinitely many solitons accompanied by radiation. These are
two distinct classes of infinite soliton configurations.

Another limit of finite multisolitons leads to soliton gases.
Here the negative eigenvalues fill fixed intervals away from
zero, and local limits can have nondecaying backgrounds.
Soliton gases describe configurations in which the number of
solitons grows with the observation length and a spectral
density replaces a finite list of eigenvalues. Their study
includes both kinetic equations and limits of exact solutions.
Zakharov \cite{gasZak71} introduced a kinetic description of
dilute KdV gases. El \cite{gasEl03} derived their
density-of-states equation from a thermodynamic limit of the
KdV--Whitham equations, and El and Kamchatnov \cite{gasEK05}
developed kinetic equations for dense KdV and nonlinear
Schr\"odinger gases through pairwise phase shifts. El and
Tovbis \cite{gasET20} related the spectral data of focusing
nonlinear Schr\"odinger gases to finite-gap limits and
nonlinear dispersion relations.

For KdV, Girotti, Grava, Jenkins and McLaughlin
\cite{gasGGJM21} constructed a deterministic gas by letting
simple eigenvalues fill an interval and prescribing norming
coefficients through a positive function analytic near that
interval. The limiting Riemann--Hilbert problem yields a
cnoidal background at one spatial end, exponential decay at
the other, and long-time asymptotics. For modified KdV,
Girotti, Grava, Jenkins, McLaughlin and Minakov
\cite{gasGGJMM23} used Fredholm determinants and
Riemann--Hilbert analysis to study the gas and its interaction
with a trial soliton. For BO, Altschul \cite{gasAltschul19}
studied the statistical equilibrium of a dilute gas, while
Gassot, G\'erard and Miller
\cite[Remark~20]{GGMExplicit26} described a deterministic
gas with negative eigenvalues accumulating at zero.
The multiphase solutions of Dobrokhotov and Krichever
\cite{gasDK91} provide a further connection between
nondecaying BO waves and prescribed spectral bands.

These developments raise several questions. Gassot and
G\'erard \cite[Section~1 and Remark~4.4]{GG} ask whether
trace equality determines the infinite-order multisoliton
class among real $L^2$ potentials with infinitely many negative
Lax eigenvalues, and which asymptotic soliton parameters can
be prescribed. For rough integrable
data, can the complete soliton series and free radiation
describe the solution in the global Sobolev norm throughout
the sharp well-posedness space? For nondecaying configurations,
can finite multisolitons converge to gases with prescribed
negative spectral bands, and can continuous norming data be
recovered from their long-time behavior?

In this work, we address these open problems through the Lax
operator and G\'erard's explicit formula. Trace variation and
Hardy model spaces recover the initial pole representation,
while finite inverse reconstruction and compactness realize
the asymptotic soliton parameters. Conserved spectral
measures, distorted Fourier analysis and averaged
moving-resolvent estimates give soliton resolution for real
$H^s\cap L^1$ data, $s>-\frac12$, with the complete soliton
series and free radiation in both time directions. In this
class, infinitely many negative eigenvalues require nonzero
radiation. For soliton gases, local limits and a continuous
action on the norming data yield finite-band and one-sided
families. Their physical Lax spectra and spatial and
long-time asymptotics determine the periodic background,
the algebraic vacuum tail and the response to the norming
weight. We now introduce the notation and state the main results.

\subsection{Main results}\label{sec:main-results}\label{subsec:main-results}
We write the BO solitary waves as $R_p(x-c_pt)$, where
$$
 R_p(x)=\frac{2\Ima p}{|x+p|^2},\enspace
 c_p=\frac1{\Ima p},\enspace p\in\C_+:=\{z:\Ima z>0\}.
$$

We use the Fourier transform
$\widehat f(\xi)=\int_\R e^{-ix\xi}f(x)\,dx$ and the inner product
$\langle f,g\rangle=\int_\R f\overline g$, linear in the first
argument. Let $\Pi$ be the orthogonal projection onto the Hardy
space
$$
 \Hp=\{f\in L^2(\R):\operatorname{supp}\widehat f
 \subset[0,\infty)\},\enspace
 H_+^r=\{f\in H^r(\R):\operatorname{supp}\widehat f
 \subset[0,\infty)\}.
$$
We also use $\Pi$ on $H^r$ as the Fourier multiplier
$\widehat{\Pi f}(\xi)=\mathbf1_{[0,\infty)}(\xi)\widehat f(\xi)$.
For a real potential $v$, write
$L_v=-i\partial_x-\Pi(v\,\cdot)$ for the BO Lax operator.
For bounded $v$ its domain is $H_+^1$, and for real $v\in L^2$
it is the self-adjoint operator associated with the closed form
$$
 \mathfrak q_v[f]
 =\||D_x|^{1/2}f\|_2^2-\int_\R v|f|^2,
 \enspace f\in H_+^{1/2}.
$$
For $v\in H^s$, $-\frac12<s<0$, the same notation refers to the
self-adjoint form realization, with the potential term interpreted
by Sobolev duality; see \cite{KLV24}. We denote its spectrum
by $\sigma(L_v)$.
The coordinate operator $X^*$ is given by
$$
 \widehat{X^*f}(\xi)=i\partial_\xi\widehat f(\xi),\enspace \xi>0,
 \enspace
 \Dom(X^*)=\{f\in\Hp:\widehat f|_{(0,\infty)}\in H^1(0,\infty)\}.
$$
For $h\in H^s$, the free BO group is defined by
\begin{equation}\label{eq:free-BO-introduction}
 \widehat{e^{t\partial_x|D_x|}h}(\xi)
 =e^{it\xi|\xi|}\widehat h(\xi).
\end{equation}
The norms $\|\cdot\|_2$ and $\|\cdot\|_\infty$ are taken on $\R$.

Let $\mathfrak P_{\rm in}$ consist of sequences
$p=(p_m)_{m\ge1}\in\C_+^{\N}$, counted with multiplicity, such that
\begin{equation}\label{eq:GG-input-introduction}
 \sum_{m,n\ge1}\frac{\Ima p_m}{|p_m-\overline{p_n}|^2}<\infty.
\end{equation}
Equivalently, the series
\begin{equation}\label{eq:infinite-multisoliton-introduction}
 u[p]=\sum_{m\ge1}R_{p_m}
\end{equation}
converges in $L^2(\R)$. For such data, write the negative eigenvalues of $L_{u[p]}$
as $-a_j[p]$, with $a_1[p]>a_2[p]>\cdots>0$, and choose
corresponding normalized eigenfunctions $\varphi_j\in\Dom(X^*)$ of $L_{u[p]}$. Set
\[
 \mathcal A(p)=(p_j^\infty[p])_{j\ge1},\enspace
 p_j^\infty[p]=-\langle X^*\varphi_j,\varphi_j\rangle,
 \enspace \|\varphi_j\|_2=1.
\]
These are the asymptotic soliton parameters of \cite{GG}.
Define the admissible output class by
\[
 \mathfrak P_{\rm out}=\left\{w\in\C_+^{\N}:
 0<\Ima w_1<\Ima w_2<\cdots,\enspace
 \sum_{j\ge1}(\Ima w_j)^{-1}<\infty\right\}.
\]

For the resolution theorem, let $u_0\in H^s\cap L^1$ be real,
$s>-\frac12$, and write the negative eigenvalues of $L_{u_0}$ as
$\lambda_j=-a_j$, $a_1>a_2>\cdots>0$, $1\le j\le N$,
where $N\in\N\cup\{0,\infty\}$. Choose eigenfunctions with
$L_{u_0}\varphi_j=-a_j\varphi_j$ and $\|\varphi_j\|_2=1$.
Lemma~\ref{rad:lem-coordinate} gives $\varphi_j\in\Dom X^*$.
We set
$$
 p_j=-\langle X^*\varphi_j,\varphi_j\rangle,\enspace
 \Ima p_j=(2a_j)^{-1},\enspace c_{p_j}=2a_j.
$$
An empty soliton sum is understood to be zero.

Our first theorem recovers the initial poles from trace equality
and realizes all admissible asymptotic soliton parameters. Write
$\Ran\mathcal A=\{\mathcal A(p):p\in\mathfrak P_{\rm in}\}$.
For $a_j>0$ and $c_j\in\R$, the asymptotic parameter
$w_j=-c_j+i/(2a_j)$ corresponds to the soliton
$R_{w_j}(x-2a_jt)$, whose velocity is $2a_j$ and whose center
is $2a_jt+c_j$.
\begin{theorem}[The inverse spectral problem for infinite-order multisolitons]
\label{thm:inverse-main}\leavevmode
\begin{enumerate}[label=\textup{(\roman*)},ref=\thetheorem(\roman*)]
\item\label[theorem]{thm:main}\label[theorem]{thm:main-gerard}
Let $u_0\in L^2(\R;\R)$. Suppose that $L_{u_0}$ has a sequence
of negative eigenvalues $(-a_k)_{k\ge1}$ with
$a_1>a_2>\cdots>0$ satisfying
\begin{equation}\label{eq:saturation}
 2\pi\sum_{k\ge1}a_k=\|\Pi u_0\|_2^2<\infty.
\end{equation}
Then these eigenvalues exhaust the negative spectrum of
$L_{u_0}$, and there exists $p^{\rm in}\in\mathfrak P_{\rm in}$
such that
$$
 u_0=u[p^{\rm in}]=\sum_{m\ge1}R_{p_m^{\rm in}}.
$$
The series converges in $L^2(\R)$ and absolutely and uniformly
on $\R$.

\item\label[theorem]{thm:exact-range}
Let $(a_j)_{j\ge1}$ be a strictly decreasing positive sequence
with $\sum_ja_j<\infty$, and let $(c_j)_{j\ge1}$ be any real
sequence. Then there exists $p^{\rm in}\in\mathfrak P_{\rm in}$
such that the complete negative spectrum of $L_{u[p^{\rm in}]}$
is $(-a_j)_{j\ge1}$ and
$$
 \mathcal A(p^{\rm in})_j=-c_j+\frac{i}{2a_j},\enspace j\ge1.
$$
Equivalently,
\begin{equation}\label{eq:exact-output-range}
 \Ran\mathcal A=\mathfrak P_{\rm out}.
\end{equation}
\end{enumerate}
\end{theorem}
\begin{re}
The trace identity in \cite{GG} is established for potentials
already represented as infinite-order multisolitons.
Theorem~\ref{thm:inverse-main}\textup{(i)} proves its converse,
recovering the initial poles, their summability and the
convergence of the soliton series. Thus, among real $L^2$
potentials with infinitely many negative Lax eigenvalues,
trace equality characterizes the infinite-order multisoliton
class. This solves the open inverse spectral problems posed in
\cite[Remark~4.4]{GG}.
\end{re}

Having proved the converse of the trace identity and determined
the exact range of the asymptotic soliton parameters, we now
turn to soliton resolution in the sharp well-posedness space
$s>-\frac12$. For real initial data
$u_0\in H^s(\R)\cap L^1(\R)$ with infinitely many negative
Lax eigenvalues, the next theorem establishes an asymptotic
decomposition in $H^s$ into an infinite soliton series and
a nonzero free radiation term as $t\to\pm\infty$.

\begin{theorem}[Soliton resolution conjecture with infinite-order multisolitons]
\label{thm:main-resolution}
Let $u_0\in H^s(\R)\cap L^1(\R)$, $s>-\frac12$.
Suppose that $L_{u_0}$ has infinitely many eigenvalues
$(-a_j)_{j\ge1}$, with $a_1>a_2>\cdots>0$.
Then there exist unique nonzero functions
$u_\infty^\pm\in H^s(\R)$ such that
\begin{equation}\label{eq:resolution-infinite}
 \lim_{t\to\pm\infty}
 \left\|u(t)-\sum_{j=1}^{\infty}R_{p_j}(\,\cdot-2a_jt)
       -e^{t\partial_x|D_x|}u_\infty^\pm\right\|_{H^s}=0.
\end{equation}
\end{theorem}
\begin{re}
	The $L^1$ assumption is scaling invariant and sharp for Theorem~\ref{thm:main-resolution} in the
	Lebesgue scale. It cannot be replaced by $L^p$ for any
	$p>1$, or by weak $L^1$, even for smooth initial data.
	In fact, consider
	$$
	u_0(x)=\sum_{m\ge1}R_{i2^m}(x)
	=\sum_{m\ge1}\frac{2^{m+1}}{x^2+4^m}.
	$$
	This function belongs to every $H^r(\R)$, $r\ge0$.
	Splitting the sum at $2^m=|x|$ gives
	$u_0(x)\asymp(1+|x|)^{-1}$, so
	$$
	u_0\in\bigcap_{1<p<\infty}L^p(\R)
	\cap L^{1,\infty}(\R),
	\enspace u_0\notin L^1(\R).
	$$
	It is an infinite-order multisoliton with infinitely many
	negative eigenvalues and an ordered $L^2$ soliton decomposition
	without radiation \cite{GG}. If the theorem held for this datum
	at $s=0$, the asserted absolute convergence, with tails uniform
	in time, would turn that ordered decomposition into the
	complete-series limit with $u_\infty^\pm=0$. This contradicts
	the nonvanishing conclusion for $N=\infty$.
\end{re}
\begin{re}
The soliton resolution result in \cite[Theorem~1.1]{GGM} assumes
$u_0,xu_0\in H^1$ and $x^2u_0=c_0+v_0$, with $v_0\in L^2$
and $c_0\in\R$, and yields finitely many solitons and free
radiation in $H^1$.
Theorem~\ref{thm:main-resolution} assumes $u_0\in H^s\cap L^1$,
$s>-\frac12$, with no spatial moment condition and permits
infinitely many solitons. Unlike the ordered $L^2$ limit for
the pole class in \cite{GG}, \eqref{eq:resolution-infinite}
uses the complete series: Lemma~\ref{rad:lem-moments} gives
$\sum_j a_j^\gamma<\infty$ for every $\gamma>0$.
Proposition~\ref{rad:prop-example} supplies smooth integrable
data with infinite negative spectrum.
\end{re}

Our third theorem concerns soliton gases obtained by filling
fixed negative spectral bands. We introduce the notation for
the two families in its statement.
We write $\mathbb T=\R/(2\pi\Z)$, denote by $\delta_\lambda$
the unit point mass at $\lambda$, and use $\mathbf1_I$ for
the indicator of an interval $I$. Convergence of measures
denoted by $\rightharpoonup$ is tested against continuous
compactly supported functions. Convergence in
$C^\infty_{\rm loc}(\R_t\times\R_x)$ means local uniform
convergence of every mixed $t,x$ derivative. In the finite-band
approximation below, $R>0$ is the spatial truncation radius,
$N_{R,\Phi}$ is the number of solitons in $u_{R,\Phi}$, and
$\lambda_{k,R,\Phi}$ are its negative Lax eigenvalues.

For the one-sided family, given $0<a<b$ and
$r\in C([a,b];(0,\infty))$, set
$$
 K=b-a,\enspace d=\frac{2\pi}{K},\enspace
 \rho=\sqrt{\frac ab},\enspace
 \gamma_r(s)=\frac{\log r(s)}{2s},\enspace s\in[a,b].
$$
For each integer $N\ge1$, define the Hermitian matrix $A_N$
and the diagonal matrix $Y_N$, with indices $0\le j,k<N$, by
$$
 (A_N)_{jj}=\frac{a+b}{2},\enspace
 (A_N)_{jk}=-\frac{i}{d(j-k)}\enspace(j\ne k),\enspace
 Y_N=d\,\diag(0,\ldots,N-1).
$$
Let $I_N$ be the identity matrix. The bounds $aI_N\le A_N\le bI_N$
ensure that $A_N^{-1}$ and $\gamma_r(A_N)$ are defined, the latter
by the spectral calculus. Writing $\Tr$ for the matrix trace, set
$$
 u_{r,N}(t,x)=-2\Ima\Tr
 \left[xI_N+Y_N-\gamma_r(A_N)-2tA_N
              +\frac{i}{2}A_N^{-1}\right]^{-1}.
$$
Let $a_{j,N}$, $1\le j\le N$, denote the eigenvalues of $A_N$.
The inverse spectral center at the eigenvalue $-a_{j,N}$ of
$L_{u_{r,N}(0)}$ is denoted by $c_{j,N}$; it is the real part of
the expectation of $X^*$ in a corresponding normalized eigenfunction. We use the norming convention
$\nu_{j,N}^{(r)}=e^{2a_{j,N}c_{j,N}}$ and denote the gas limit
asserted below by $u_r$, with $u_1$ corresponding to $r\equiv1$.
The periodic background and its phase are
$$
 \delta_r=\frac1K\int_a^b\gamma_r(s)\,ds,\enspace
 Q_r(t,x)=K\frac{1-\rho^2}
 {1-2\rho\cos\!\bigl(Kx-K(a+b)t-K\delta_r\bigr)+\rho^2}.
$$
For the large-time asymptotics, we use $\xi=x/t$ and the functions
$$
 \Phi_{\rm ray}(\xi)=\frac1\pi
 \operatorname{arctanh}^{2}\!\frac{b-a}{a+b-\xi},\enspace\xi<2a,
 \enspace
 F(\xi)=\frac1{2\pi\xi}
 \log\frac{b(\xi-2a)}{a(\xi-2b)},\enspace\xi>2b,
$$
where $\operatorname{arctanh}z=\tfrac12\log((1+z)/(1-z))$
for $0<z<1$.
\begin{theorem}[Soliton gases and their asymptotics]\label{thm:main-gas}\leavevmode
\begin{enumerate}[label=\textup{(\roman*)}]
\item
Let $m\ge1$ be an integer and
$p_1<q_1<\cdots<p_m<q_m<0$.
For every $\Phi\in\mathbb T^m$, there is a bounded smooth
quasiperiodic solution $u_\Phi$ of \eqref{eq:BO-introduction}
such that, for every $t\in\R$,
$$
 \sigma(L_{u_\Phi(t)})=\bigcup_{j=1}^m[p_j,q_j]\cup[0,\infty),
$$
and this spectrum is purely absolutely continuous of
multiplicity one. There are finite multisoliton solutions
$u_{R,\Phi}$ with simple negative Lax eigenvalues such that,
as $R\to\infty$,
$$
 u_{R,\Phi}\longrightarrow u_\Phi
 \enspace\hbox{in }C^\infty_{\rm loc}(\R_t\times\R_x),\enspace
 \frac1{2R}\sum_{k=1}^{N_{R,\Phi}}\delta_{\lambda_{k,R,\Phi}}
 \rightharpoonup\frac1{2\pi}\sum_{j=1}^m
 \mathbf1_{[p_j,q_j]}(\lambda)\,d\lambda.
$$
Both limits are uniform in $\Phi$.

\item
Let $0<a<b$ and $r\in C([a,b];(0,\infty))$.
For every integer $N\ge1$, the functions $u_{r,N}$ are
$N$-soliton solutions with distinct negative Lax eigenvalues
$-a_{j,N}$, where $a_{j,N}\in[a,b]$, and inverse spectral data
$$
 c_{j,N}=-\frac{d(N-1)}2+\gamma_r(a_{j,N}),\enspace
 \nu_{j,N}^{(r)}=e^{-d(N-1)a_{j,N}}r(a_{j,N}),\enspace 1\le j\le N.
$$
As $N\to\infty$,
$$
 u_{r,N}\longrightarrow u_r
 \enspace\hbox{in }C^\infty_{\rm loc}(\R_t\times\R_x),\enspace
 \frac1N\sum_{j=1}^N\delta_{a_{j,N}}
 \rightharpoonup K^{-1}\mathbf1_{[a,b]}(s)\,ds.
$$
The limit $u_r$ is a global smooth BO solution with
$0\le u_r\le4b$. For every $t\in\R$,
$\sigma(L_{u_r(t)})=[-b,-a]\cup[0,\infty)$, and its negative
spectrum is purely absolutely continuous of multiplicity one.

For every $T_0>0$, uniformly for $|t|\le T_0$,
$$
 \begin{aligned}
 u_r(t,x)&=Q_r(t,x)+o(1),&&x\to-\infty,\\
 u_r(t,x)&=\frac{\log(b/a)}{2\pi x}
 +O\!\left(\frac{\log(2+x)}{x^2}\right),&&x\to+\infty.
 \end{aligned}
$$
As $t\to+\infty$,
$$
 \begin{aligned}
 u_r(t,t\xi)&=Q_r\!\left(t,t\xi-\frac{\Phi_{\rm ray}(\xi)}K\right)+o(1),
 &&\xi<2a,\\
 u_r(t,t\xi)&=\frac{F(\xi)}t+O(t^{-3/2}),
 &&\xi>2b,
 \end{aligned}
$$
uniformly for $\xi$ in compact subsets of the respective regions.
Moreover,
$$
 t^2(u_r-u_1)(t,t\xi)\longrightarrow
 \frac1{2\pi}\int_a^b\frac{\gamma_r(s)}{s(\xi-2s)^2}\,ds
 \enspace\hbox{as }t\to+\infty,
$$
uniformly on compact subsets of $(2b,\infty)$.
These limits for all $\xi>2b$ determine $r$ uniquely, and
$r\mapsto u_r(0,\cdot)$ is injective.
\end{enumerate}
\end{theorem}
\begin{re}
The KdV construction in \cite{gasGGJM21} uses a norming function
analytic near the spectral interval and has an exponentially
decaying vacuum tail. The BO family in
Theorem~\ref{thm:main-gas}\textup{(ii)} admits every strictly
positive continuous weight $r$ and has an algebraic $1/x$
vacuum tail. Both gases approach a periodic background at the
opposite spatial end. For BO, the large-time profiles are
determined in $x/t<2a$ and $x/t>2b$. The periodic background
depends on $r$ through $\delta_r$, and the leading exterior
term is independent of $r$. The limits of
$t^2(u_r-u_1)(t,t\xi)$ for all $\xi>2b$ determine the whole
weight.
\end{re}

\subsection{Plan of the proofs}\label{subsec:proof-plan}
For Theorem~\ref{thm:inverse-main}, variation of the trace
inequality yields a local identity. The Lax--shift relations
imply backward-shift invariance, and model-space factorization
recovers the initial poles. Finite inverse spectral formulas,
compactness of eigenfunctions with bounds for $-i\partial_x$
and $X^*$, and equality in Bessel's inequality then realize
the prescribed asymptotic parameters. For
Theorem~\ref{thm:main-resolution}, the argument is organized
around G\'erard's formula under rough approximation, coordinate
estimates for negative eigenfunctions, and distorted Fourier
analysis of the radiation. The discrete, threshold, continuous,
and high-energy components enter the global $H^s$ decomposition,
with the negative spectral cutoff removed after the large-time
limit. For Theorem~\ref{thm:main-gas}, we establish continuity
of the local BO flow and convergence of localized spectral
traces. Inner functions and the continuous norming action give
the finite-band approximations and their spectra. A Toeplitz
representation and an intertwining map identify the physical
negative spectrum of the one-sided family. Resolvent estimates
and the spectral analysis of a step operator give its spatial
limits and the periodic and exterior long-time profiles.

The structure of this work is as follows.
Sections~\ref{sec:lax-data}--\ref{sec:poisson-gram} prove
the trace characterization, and
Section~\ref{sec:exact-range-proof} establishes the inverse
realization theorem. Section~\ref{sec:quantitative} gives
remainder estimates for the trace-saturated class and their
algebraic-rate consequences. Sections~\ref{gas:sec-phase-space}--\ref{gas:continuous-norming-section}
develop the local BO flow, finite-band gases and continuous
norming data. Sections~\ref{gas:one-sided-continuum}--\ref{gas:one-sided-gas-ray-section}
construct the one-sided gas and determine its spectrum and
asymptotics.
\begin{Acknowledgment}
This work was supported by the National Natural Science Foundation of China under Grant No. 12371255, and the Fundamental Research Funds for the Central Universities of CUMT under Grant No. 2024ZDPYJQ1003.
\end{Acknowledgment}
\allowdisplaybreaks[2]
\setlength{\abovedisplayskip}{2pt plus 1pt minus 1pt}
\setlength{\belowdisplayskip}{2pt plus 1pt minus 1pt}
\setlength{\abovedisplayshortskip}{0pt}
\setlength{\belowdisplayshortskip}{2pt plus 1pt}
\setlength{\jot}{2pt}
\makeatletter
\def\thm@space@setup{\thm@preskip=2pt plus 1pt
  \thm@postskip=2pt plus 1pt}
\renewenvironment{proof}[1][\proofname]{\par
  \pushQED{\qed}\normalfont\topsep2pt plus 1pt\relax
  \trivlist\item[\hskip\labelsep\itshape #1\@addpunct{.}]\ignorespaces
}{\popQED\endtrivlist\@endpefalse}
\makeatother
\section{The Lax operator and the trace inequality}
\label{sec:lax-data}\label{sec:sharp-trace}
Use the notation of Subsection~\ref{subsec:main-results}, and set
$D=-i\partial_x$ and $I_+f=\widehat f(0+)$.
For real $v\in L^2$, the sesquilinear form of $L_v$ is
$\mathfrak l_v[f,g]=\langle D^{1/2}f,D^{1/2}g\rangle-\int vf\overline g$
on $H^{1/2}_+$.
Write $q_v=\Pi v$, $P_{\rm d}(v)=E_{L_v}(( -\infty,0))$,
$\Hd(v)=\Ran P_{\rm d}(v)$, and $\Hc(v)=\Hd(v)^\perp$.
The Hardy shifts are $S_tf=e^{\ii tx}f$ and
$(S_t^*f)^{\wedge}(\xi)=\widehat f(\xi+t)$, $t,\xi\ge0$;
$K_\Theta=\Hp\ominus\Theta\Hp$ denotes a model space.

\begin{prop}\label{thm:GG-infinite-resolution}
For $p\in\mathfrak P_{\rm in}$, $u_0=u[p]$ belongs to every $H^r$, $r\ge0$.
Its negative Lax eigenvalues $\lambda_1<\lambda_2<\cdots<0$ are simple.
For normalized eigenfunctions $\varphi_j\in\Dom X^*$, the parameters
$p_j^\infty=-\langle X^*\varphi_j,\varphi_j\rangle$ satisfy
$$
 \Ima p_j^\infty=(2|\lambda_j|)^{-1},\enspace
 \sum_j(\Ima p_j^\infty)^{-1}=\pi^{-1}\|\Pi u_0\|_2^2.
$$
\end{prop}
These assertions follow from \cite[Theorem~1.3, Propositions~3.1 and~4.3,
Lemma~5.3]{GG} and \cite[Theorem~1.4]{BKV26}. For real $v\in L^2$,
\cite[Theorem~2.3]{BKV26} and the Wu relation give
$\sigma_{\rm ess}(L_v)=[0,\infty)$, $|\langle\Pi v,\psi\rangle|^2=2\pi b\|\psi\|_2^2 \enspace(L_v\psi=-b\psi,\ b>0)$.
For example, $p_m=i2^m$ belongs to $\mathfrak P_{\rm in}$, since
$\sum_n2^m/(2^m+2^n)^2\le(m+1/3)2^{-m}$.
The preceding proposition therefore gives a nonzero potential with
infinitely many negative eigenvalues and trace equality.
Until Section~\ref{sec:exact-range-proof}, let $u_0$, $(a_k)$ and the
normalized eigenfunctions $(\varphi_k)$ satisfy the hypotheses of Theorem~\ref{thm:main}.
Put $L=L_{u_0}$, $q=\Pi u_0$, and choose phases so that
$q=\sum_k\sqrt{2\pi a_k}\,\varphi_k$ by equality in Bessel's inequality.
The Wu relation excludes further negative eigenfunctions. Thus
$\Hd=\overline{\spn\{\varphi_k\}}$, $P_d=E_L(( -\infty,0))$,
$P_c=I-P_d$, and $L_d=-A$, where $A\varphi_k=a_k\varphi_k$ and
$A\in\Sch_1$. For $L_c=L|_{\Hc}$,
\begin{equation}\label{eq:Lc-nonnegative}
 \sigma(L_c)\subset[0,\infty),\enspace
 I_{\Hc}=E_{L_c}((0,\infty))+E_{L_c}(\{0\}).
\end{equation}

We next examine equality in the trace inequality.

\begin{lemma}
For real $v\in L^2$, let $(-b_j)$ be all negative eigenvalues of $L_v$,
with normalized eigenfunctions $\psi_j$. Then
\begin{equation}\label{eq:sharp-trace}
 2\pi\Tr(-L_v)_+=2\pi\sum_jb_j
 =\sum_j|\langle\Pi v,\psi_j\rangle|^2
 \le\|\Pi v\|_2^2=\tfrac12\|v\|_2^2.
\end{equation}
For every finite orthonormal family $(e_\ell)\subset H^{1/2}_+$,
\begin{equation}\label{eq:Ky-Fan-form}
 \sum_\ell-\mathfrak l_v[e_\ell]
 \le\sum_\ell\langle(-L_v)_+e_\ell,e_\ell\rangle
 \le\Tr(-L_v)_+.
\end{equation}
\end{lemma}
\begin{proof}
Apply the Wu relation and Bessel's inequality for \eqref{eq:sharp-trace},
and the spectral theorem for \eqref{eq:Ky-Fan-form}.
\end{proof}

\begin{lemma}
Under \eqref{eq:saturation},
\begin{equation}\label{eq:local-trace}
 u_0=2\pi\sum_k|\varphi_k|^2\enspace\text{in }L^1_{\rm loc}\text{ and a.e.}
\end{equation}
\end{lemma}
\begin{proof}
For $0\le w\in L_c^\infty$, apply
\eqref{eq:sharp-trace}--\eqref{eq:Ky-Fan-form} to $u_0\pm\eps w$ and
$\varphi_1,\ldots,\varphi_N$. The plus sign first gives local summability.
Letting $N\to\infty$ and subtracting \eqref{eq:saturation} gives
$\pm2\pi\sum_k\int w|\varphi_k|^2\le\pm\int u_0w+\frac\eps2\|w\|_2^2$.
Let $\eps\to0$.
\end{proof}

\begin{prop}\label{prop:zero-by-saturation}
Under \eqref{eq:saturation}, $\Ker L=\Ker L_c=\{0\}$.
\end{prop}
\begin{proof}
If $L\psi=0$, $\|\psi\|_2=1$, include $\psi$ in the preceding
orthonormal family for $u_0+\eps w$. After $N\to\infty$,
\eqref{eq:local-trace} yields
$\int w|\psi|^2\le\eps\|w\|_2^2/(4\pi)$, hence $\psi=0$.
\end{proof}

\section{Lax shifts and inverse reconstruction}
\label{sec:lax-shift}\label{sec:poisson-gram}\label{sec:model-space}
\begin{lemma}
For real $u_0\in L^2$, $\Dom L=H^1_+$ with equivalent graph norm,
and both Hardy shifts preserve this domain. For $f,g\in H^{1/2}_+$,
\begin{equation}\label{eq:shift-form}
 \mathfrak l_{u_0}[S_tf,S_tg]=\mathfrak l_{u_0}[f,g]+t\langle f,g\rangle.
\end{equation}
For $f\in H^1_+$,
\begin{equation}\label{eq:shift-Lax}
 S_t^*LS_tf=(L+t)f,\enspace
 S_t^*Lf=(L+t)S_t^*f-\mathcal B_tf,
\end{equation}
where, for every $f\in\Hp$,
\begin{equation}\label{eq:B-formula}
 \mathcal B_tf=\frac1{2\pi}\int_0^t\widehat f(s)S_{t-s}^*q\,ds,
 \enspace \|\mathcal B_tf\|_2\le\sqrt{\frac t{2\pi}}\|q\|_2\|f\|_2.
\end{equation}
If $q\in\Hd$, set $\mathcal F_t=P_cS_t^*P_d$. Then
$\mathcal F_t\varphi_k\in\Dom L_c$ and
\begin{equation}\label{eq:cross-Volterra}
 (L_c+t+a_k)\mathcal F_t\varphi_k
 =\frac1{2\pi}\int_0^t\widehat{\varphi_k}(s)\mathcal F_{t-s}q\,ds.
\end{equation}
\end{lemma}
\begin{proof}
The resolvent estimate \cite[Theorem~2.3]{BKV26} gives
$\Dom L\subset H^1_+$; the reverse inclusion follows from $u_0f\in L^2$
for $f\in H^1_+$. The closed graph theorem gives the norm equivalence.
Fourier translation on a core and graph-norm density prove
\eqref{eq:shift-form}--\eqref{eq:B-formula}; Plancherel gives the bound.
Apply the second identity in \eqref{eq:shift-Lax} to $\varphi_k$ and project
onto $\Hc$ to obtain \eqref{eq:cross-Volterra}.
\end{proof}

\begin{prop}\label{prop:fixed-band}
Under \eqref{eq:saturation}, for $t\ge0$ and $k\ge1$,
$E_{L_c}((0,\infty))\mathcal F_tq =E_{L_c}((0,\infty))\mathcal F_t\varphi_k=0$.
\end{prop}
\begin{proof}
Represent $L_c$ in a direct integral as multiplication by $\lambda$,
restrict to $\Delta=[\lambda_0,R]\Subset(0,\infty)$, and write
$g(t,\lambda)=(UE_{L_c}(\Delta)\mathcal F_tq)(\lambda)$,
with norms in the spectral fibres and integrals against the spectral measure. Summing \eqref{eq:cross-Volterra} with coefficients $\sqrt{2\pi a_k}$ gives
\begin{equation}\label{eq:fixed-band-limit}
 g(t,\lambda)=\int_0^tK(t,\lambda;s)g(t-s,\lambda)\,ds,
 \enspace K(t,\lambda;s)=\frac1{2\pi}
 \widehat{(A+\lambda+t)^{-1}q}(s).
\end{equation}
Put $q_{\le N}=\sum_{k\le N}\sqrt{2\pi a_k}\varphi_k$.
The truncated sums converge in $L^2([0,T]\times\Delta)$:
Plancherel bounds $\|K(t,\lambda;\cdot)\|_2$ by
$C_\Delta=\|q\|_2/(\sqrt{2\pi}\lambda_0)$ and bounds the truncated
kernel error by $\|q-q_{\le N}\|_2/(\sqrt{2\pi}\lambda_0)$.
Cauchy--Schwarz and Fubini justify the limit.
On an interval of length $\ell$ following a zero interval,
\eqref{eq:fixed-band-limit} gives $\|g\|_{L^2}^2\le\ell C_\Delta^2\|g\|_{L^2}^2$.
Choose $\ell C_\Delta^2<1$ and iterate. Strong continuity and exhaustion by
$[1/n,n]$ give the assertion for $q$; \eqref{eq:cross-Volterra} gives it
for each $\varphi_k$.
\end{proof}

\begin{coro}
Under \eqref{eq:saturation},
\begin{equation}\label{eq:cross-zero}
 P_cS_t^*P_d=0,\enspace S_t^*\Hd\subset\Hd\enspace(t\ge0).
\end{equation}
\end{coro}
\begin{proof}
Combine Propositions~\ref{prop:fixed-band} and~\ref{prop:zero-by-saturation}
with \eqref{eq:Lc-nonnegative}, then use density of $(\varphi_k)$ in $\Hd$.
\end{proof}

\subsection{Proof of Theorem~\ref{thm:inverse-main}\textup{(i)}: the reconstruction of infinite-order multisolitons}
\label{subsec:proof-inverse-trace}
\begin{proof}
Put $\rho=u_0/(2\pi)$ and let $P_y$ be the Poisson extension.
Since $\Hd\ne0$, $\Hc\ne\Hp$. If $\Hc=0$, the Hardy kernel and
\eqref{eq:local-trace} give
$(4\pi y)^{-1}=\sum_k|\varphi_k(x+\ii y)|^2\le P_y\rho(x)$,
contradicting $y\to0$ at a Lebesgue point.
By \eqref{eq:cross-zero}, $\Hc$ is forward-shift invariant; Beurling's
theorem \cite[Theorem~4.3]{GarciaModel} gives $\Hd=K_\Theta$ for a
nonconstant inner function $\Theta$. Its diagonal kernel satisfies
$$
 \frac{1-|\Theta(x+\ii y)|^2}{4\pi y}
 =\sum_k|\varphi_k(x+\ii y)|^2\le P_y\rho(x).
$$
Taking boundary limits of finite sums gives the reverse lower bound.
Hence the left side tends to $\rho$ almost everywhere, and
Julia--Carath\'eodory gives $-\ii\Theta'/\Theta=u_0$ there.
The canonical factorization and angular-derivative formula
\cite{Garnett,AhernClark} therefore yield
\begin{equation}\label{eq:phase-density}
 u_0(x)=\beta+\sum_m\frac{2\Ima q_m}{|x-q_m|^2}
       +\int_\R\frac{1+t^2}{(x-t)^2}\,d\sigma(t)\enspace\text{a.e.},
\end{equation}
where $\beta\ge0$, $(q_m)\subset\C_+$ are the zeros, and $\sigma$ is
a finite positive singular measure. Positivity and $u_0\in L^2$ force
$\beta=0$. If $\sigma\ne0$, choose a compact set of positive $\sigma$-mass
inside a bounded interval $J$. Since
$\int_J|x-t|^{-2}\,dx=\infty$, Tonelli contradicts $u_0\in L^1(J)$.
Thus $\Theta$ is a pure Blaschke product, with infinitely many zeros
because $\dim K_\Theta=\infty$.

Set $p_m^{\rm in}=-\overline{q_m}$. Equation~\eqref{eq:phase-density}
gives $u_0=\sum_mR_{p_m^{\rm in}}$; monotonicity and domination by $u_0$
give $L^2$ convergence. The finite Gram identity
\cite[Lemma~4.1, (4.2)]{GG} and monotone convergence give
\begin{equation}\label{eq:infinite-Gram}
 4\pi\sum_{m,n}\frac{\Ima p_m^{\rm in}}
 {|p_m^{\rm in}-\overline{p_n^{\rm in}}|^2}
 =\|\Pi u_0\|_2^2=2\pi\sum_ka_k.
\end{equation}
Its diagonal bounds $\pi\sum_m(\Ima p_m^{\rm in})^{-1}$.
Since $0\le R_p\le2/\Ima p$, the series also converges absolutely and
uniformly, and $p^{\rm in}\in\mathfrak P_{\rm in}$.
\end{proof}

\begin{coro}\label{cor:trace-saturation-classification}
For real $u\in L^2$, let $(-a_j)_{j\in J}$ be its complete negative Lax
spectrum. Then $2\pi\sum_ja_j=\|\Pi u\|_2^2$ precisely for $u=0$
if $J=\varnothing$, a finite multisoliton if $0<|J|<\infty$, or an
infinite-order multisoliton if $J=\N$.
\end{coro}
\begin{proof}
Equality in Bessel's inequality gives
$\Pi u=\sum_j\sqrt{2\pi a_j}\,\varphi_j$ after phase normalization.
For finite $J$ and sufficiently large $\kappa$, its resolvent pairing is
$\sum_j2\pi a_j/(\kappa-a_j)$, so \cite[Theorem~1.5]{BKV26} applies.
The infinite case is Theorem~\ref{thm:main}; the converse follows from
\cite[Theorems~1.4 and~1.6]{BKV26} and \cite[Proposition~4.3]{GG}.
The empty case follows from real-valuedness.
\end{proof}

\begin{lemma}\label{lem:Xstar-discrete}
Under \eqref{eq:saturation}, $\varphi_k\in\Dom X^*$ and
$X^*\varphi_k\in\Dom L\cap\Hd$.
\end{lemma}
\begin{proof}
After the preceding reconstruction, \cite[Proposition~3.1]{GG} gives
the domain assertions and
$(L+a_k)X^*\varphi_k=-\ii\varphi_k+\ii I_+(\varphi_k)q/(2\pi)$.
Projecting onto $\Hc$ and using $L_c+a_k\ge a_kI$ gives the claim.
\end{proof}
Moreover, \eqref{eq:infinite-Gram} and \cite[Lemma~4.1]{GG} imply
$u_0\in H^r$ for every $r\ge0$, so its global BO solution has the
parameters in Proposition~\ref{thm:GG-infinite-resolution}.

\section{Wu's identities and the exact range of the map}
\label{sec:exact-range-proof}
For $w\in\mathfrak P_{\rm out}$, put
$a_j=(2\Ima w_j)^{-1}$, $c_j=-\Rea w_j$, and $Q_N=Q_{\{-a_1,\ldots,-a_N\},(c_1,\ldots,c_N)}$.
The finite inverse formula \cite[Definition~1.2]{BKV26} is
$Q_N=-2\Ima\,\partial_x\log\det(xI+B_N)$, where
$(B_N)_{jj}=w_j$ and $(B_N)_{jk}=\ii/(a_j-a_k)$ for $j\ne k$.
Since $\Ima B_N=\diag((2a_j)^{-1})>0$, its eigenvalues lie in $\C_+$;
factoring the determinant expresses $Q_N$ as a positive sum of $N$
soliton profiles. Normalize $L_{Q_N}\varphi_j^{(N)}=-a_j\varphi_j^{(N)}$
so that, with $q^{(N)}=\Pi Q_N$,
\begin{equation}\label{eq:finite-q-expansion}
 q^{(N)}=\sum_{j\le N}\sqrt{2\pi a_j}\,\varphi_j^{(N)},\enspace
 \|q^{(N)}\|_2^2=\tfrac12\|Q_N\|_2^2=2\pi\sum_{j\le N}a_j.
\end{equation}
The inverse spectral matrix of \cite[Definition~5.1, Theorem~5.2,
Proposition~5.4 and Corollary~5.5]{Sun21} is $-D_NB_ND_N^{-1}$,
$D_N=\diag(\sqrt{a_j})$. Its diagonal and the Wu relations give
\begin{equation}\label{eq:finite-center-Wu}
 \langle X^*\varphi_j^{(N)},\varphi_j^{(N)}\rangle=-w_j,
 \enspace I_+\varphi_j^{(N)}=\sqrt{2\pi/a_j}.
\end{equation}
As in Lemma~\ref{lem:Xstar-discrete}, the commutator excludes a continuous
component of $X^*\varphi_j^{(N)}$; its off-diagonal entries are
\begin{equation}\label{eq:finite-X-offdiagonal}
 \langle X^*\varphi_j^{(N)},\varphi_k^{(N)}\rangle
 =\frac{\ii}{a_j-a_k}\sqrt{\frac{a_k}{a_j}}\enspace(k\ne j).
\end{equation}

\begin{lemma}\label{lem:phase-space-compactness}
Sets bounded in $\|f\|_2+\|Df\|_2+\|X^*f\|_2$ are relatively compact in $\Hp$.
\end{lemma}
\begin{proof}
Use Rellich on each Fourier interval $(0,R)$ and
$\int_R^\infty|\widehat f|^2\le R^{-2}\int_0^\infty\xi^2|\widehat f|^2$.
\end{proof}

\subsection{Proof of Theorem~\ref{thm:inverse-main}\textup{(ii)}: the exact range of the asymptotic soliton parameters}
\label{subsec:proof-inverse-range}
\begin{proof}
Proposition~\ref{thm:GG-infinite-resolution} gives
$\Ran\mathcal A\subset\mathfrak P_{\rm out}$. Conversely, fix
$w\in\mathfrak P_{\rm out}$ and set $A=\sum_ja_j<\infty$,
$U=(4\pi A)^{1/2}$. By \eqref{eq:finite-q-expansion}, a subsequence
$Q_{N_\ell}\rightharpoonup u_0$ in $L^2$, with $q^{(N_\ell)}\rightharpoonup q=\Pi u_0$.
For each fixed $j$, \eqref{eq:finite-center-Wu}--\eqref{eq:finite-X-offdiagonal}
give $\sup_{N\ge j}\|X^*\varphi_j^{(N)}\|_2<\infty$: the tail is bounded by
$4a_j^{-3}\sum_{a_k\le a_j/2}a_k$. The eigenvalue equation and
Gagliardo--Nirenberg give, for $Y_N=\|D\varphi_j^{(N)}\|_2$,
$Y_N\le a_j+U\sqrt{2Y_N}\le a_j+Y_N/2+U^2$.

Lemma~\ref{lem:phase-space-compactness} and weak graph closure yield a
diagonal subsequence with $\varphi_j^{(N_\ell)}\to\varphi_j$ strongly in
$L^2$, and $(D\varphi_j^{(N_\ell)},X^*\varphi_j^{(N_\ell)})$
weakly convergent to $(D\varphi_j,X^*\varphi_j)$.
Gagliardo--Nirenberg also gives uniform convergence; hence
$Q_{N_\ell}\varphi_j^{(N_\ell)}\rightharpoonup u_0\varphi_j$ in $L^2$.
The limits are orthonormal, satisfy $L_{u_0}\varphi_j=-a_j\varphi_j$, and obey
\begin{equation}\label{eq:limit-output-data}
 \langle q,\varphi_j\rangle=\sqrt{2\pi a_j},\enspace
 \langle X^*\varphi_j,\varphi_j\rangle=-w_j.
\end{equation}
Bessel's inequality and weak lower semicontinuity give
$2\pi A\le\|q\|_2^2\le\liminf_\ell\|q^{(N_\ell)}\|_2^2=2\pi A$.
Thus $q^{(N_\ell)}\to q$ and $Q_{N_\ell}\to u_0$ strongly in $L^2$.
The Wu relation excludes additional negative eigenvalues.
Theorem~\ref{thm:main} reconstructs $u_0=u[p^{\rm in}]$, and
\eqref{eq:limit-output-data} with \cite[Lemma~5.3]{GG} gives
$\mathcal A(p^{\rm in})=w$, proving \eqref{eq:exact-output-range}.
\end{proof}

\section{Soliton resolution for integrable initial data}
\label{sec:integrable-resolution}
Write $D=-i\partial_x$, $T_v=\Pi(v\,\cdot)$, $I_+f=\widehat f(0+)$,
$\mathcal Ff=\widehat f$, $G_z=(D-z)^{-1}\Pi$, $M_vh=vh$ and
$\mathfrak h(I)=L^2(I,dE/(2\pi))$, with $\mathfrak h_+=\mathfrak h((0,\infty))$.
For a smooth unit-interval partition $(\chi_j)$, set
$\|h\|_{H^r_{\rm uloc}}=\sup_j\|\chi_jh\|_{H^r}$.
All spectral sums include the finite and empty cases.

\subsection{Spectral source and coordinate estimates}
\begin{lemma}\label{rad:lem-source}
For the BO solution with real $u_0\in H^s$, $s>-\frac12$, there is a
common shift $A_t=L_{u(t)}+\lambda_0\ge1$ such that, uniformly in $t$,
\begin{equation}\label{rad:eq-scales}
 \|A_t^sf\|_2\asymp\|f\|_{H^s},\enspace f\in H_+^s,
\end{equation}
using completion for $s<0$. The finite measure
$\nu_s(B)=\|\one_B(L_{u(t)})A_t^s\Pi u(t)\|_2^2$ is conserved, has no
atom at zero, and satisfies
\begin{equation}\label{rad:eq-source-tail}
 \sup_t\|\eta(L_{u(t)})\Pi u(t)\|_{H^s}^2
 \le C_s\int|\eta|^2\,d\nu_s
\end{equation}
for bounded Borel $\eta$. Under smooth $H^s$ approximation, the constants
are uniform, the measures converge narrowly with their masses, and
\begin{equation}\label{rad:eq-fixed-time}
 \psi(L_{u_n(t)})\Pi u_n(t)\to\psi(L_{u(t)})\Pi u(t)
 \enspace\text{in }L^2\enspace(\psi\in C_c(\R),\ t\text{ fixed}).
\end{equation}
For $s>0$ and $K>0$, uniformly in $t$,
\begin{equation}\label{rad:eq-gain}
 g=\one_{[-K,K]}(L_{u(t)})g
 \enspace\Longrightarrow\enspace\|g\|_{H^{s+1}}\le C_{s,K}\|g\|_2.
\end{equation}
\end{lemma}
\begin{proof}
Use the flow and approximation in \cite{KLV24} and the form realization,
resolvent continuity and uniform negative Sobolev bounds in
\cite[Theorem~2.3 and Proposition~2.4]{BKV26}. This realization has
essential spectrum $[0,\infty)$ and discrete negative spectrum.
For $-\frac12<s<0$, choose
$\frac12(D+\lambda_0)\le A_t\le\frac32(D+\lambda_0)$; inversion and
L\"owner--Heinz with exponent $-2s$ give \eqref{rad:eq-scales}.
Resolvent continuity, the uniform $H^s\to L^2$ bound and density yield
$A_{v_n}^s\Pi v_n\to A_v^s\Pi v$ in $L^2$.
Functional calculus and smooth Lax covariance
(Lemma~\ref{rad:lem-observation}) now give conservation, measure
convergence and \eqref{rad:eq-source-tail}--\eqref{rad:eq-fixed-time}.

For integers $m\ge1$ and $v\in H^{m-1}$,
$Df=(L_v+\lambda_0)f-\lambda_0f+\Pi(vf)$ gives
$\Dom(L_v+\lambda_0)^m=H_+^m$ with equivalent graph norms.
At $m=1$, first use the rough resolvent to obtain
$H^{3/4}\subset L^\infty$, then absorb
$\|vf\|_2\lesssim\|v\|_2\|f\|_2^{1/2}\|Df\|_2^{1/2}$;
induct using the Sobolev algebra property and interpolate.
For $s\ge0$, the conserved measure
$d\mu(\lambda)=(\lambda+\lambda_0)^{1/2}d\nu_{-1/4}(\lambda)$
is the ordinary source measure; its mass gives the uniform $L^2$ bound.
At $t=0$, the graph-norm identities and interpolation give
$\int(\lambda+\lambda_0)^{2s}\,d\mu<\infty$.
Put $m=\lfloor s\rfloor$. For $1\le k\le m$, the graph-norm
constants for $A_t^k$ depend on the already bounded
$H^{k-1}$ norm of $u(t)$; the conserved $2k$ moment therefore
gives a uniform $H^k$ bound. The uniform $H^m$ bound controls
the graph norms of $A_t^{m+1}$, and interpolation between
$A_t^m$ and $A_t^{m+1}$ proves \eqref{rad:eq-scales} at $s$.
Its conserved $2s$ moment gives the uniform $H^s$ bound.
For $v_n\to v$ in $H^s$, these scale constants are uniform.
For fixed smooth Hardy vectors $f$, strong resolvent convergence and
$\sup_n\|A_{v_n}^{m+1}f\|_2<\infty$ give
$A_{v_n}^sf\to A_v^sf$: spectral truncation at $R$ leaves a tail
bounded by $R^{s-m-1}\|A_{v_n}^{m+1}f\|_2$.
Density in $H^s$ then gives
$A_{v_n}^s\Pi v_n\to A_v^s\Pi v$ in $L^2$.
Functional calculus gives narrow convergence of the source measures
with their masses and \eqref{rad:eq-fixed-time}. The same elliptic identity on vectors with compact spectral
support gives \eqref{rad:eq-gain}. For $f\in\ker L_{u(t)}$, Wu's
relation \cite[Theorem~1.4]{BKV26} gives
$\langle\Pi u(t),f\rangle=0$, so $\nu_s(\{0\})=0$.
Finiteness and narrow convergence then give uniform source tails
at zero and infinity.
\end{proof}

\begin{lemma}\label{rad:lem-coordinate}
Let $v\in H^\rho(\R;\R)$, $-\frac12<\rho<0$, and let $L_vf=-af$,
$a>0$, $\|f\|_2=1$. Then
\begin{equation}\label{rad:eq-coordinate}
 f\in\Dom X^*,\enspace f,X^*f\in H_+^{1+\rho},
\end{equation}
and, in $H^\rho$,
\begin{equation}\label{rad:eq-commutator}
 (L_v+a)X^*f=-if+\frac{i}{2\pi}I_+(f)\Pi v.
\end{equation}
For $c=\langle\Pi v,f\rangle$, $p=-\langle X^*f,f\rangle$,
$cI_+(f)=2\pi$, $|c|^2=2\pi a$, $|I_+(f)|^2=2\pi/a$, $\Ima p=(2a)^{-1}$.
At an isolated limiting eigenvalue, normalized eigenfunctions with
consistent phases and their $X^*$ images converge in $H^{1+\rho}$
under $H^\rho$ convergence of the potentials.
The same holds for $(L_v+\lambda_0)^{-1}\varphi$ when
$\widehat\varphi\in C_c^\infty(0,\infty)$ and $\lambda_0$ is a common
sufficiently large shift.
\end{lemma}
\begin{proof}
Put $Y=H_+^{1+\rho}$, $A_w^a=(D+a)^{-1}T_w$. Two applications of
the eigenfunction equation, starting from the form domain and using
$H^\rho H^{1/2}\subset H^{\rho-\varepsilon}$, give $f\in Y$.
Here $0<\varepsilon<\min(\frac12+\rho,1+2\rho)$, and
$\|wh\|_{H^\rho}\lesssim\|w\|_{H^\rho}\|h\|_Y$.
The Fourier kernel of $A_w^a$ is
$\widehat w(\xi-\zeta)/(2\pi(\xi+a))$, $\xi,\zeta>0$;
weighted Cauchy--Schwarz gives continuity at zero.
Split $v=z+w$ with $z$ compactly supported, $xz\in H^\rho$ and
$\|A_w^a\|_{Y\to Y}<1/2$. For
$\delta_h\psi=(\psi(\cdot+h)-\psi)/h$, direct subtraction gives
$$
 [\delta_h,\widehat A_w^a]\psi
 =-\frac{(\widehat A_w^a\psi)(\xi+h)}{\xi+a}
 +\frac{\int_0^h\widehat w(\xi+h-\zeta)\psi(\zeta)\,d\zeta}
 {2\pi h(\xi+a)}.
$$
The right side converges in $\mathcal FY$, while
$\partial_\xi\widehat A_z^a\psi
=-(\widehat A_z^a\psi)/(\xi+a)-i\widehat A_{xz}^a\psi$.
Invert $I-A_w^a$ to prove \eqref{rad:eq-coordinate}--\eqref{rad:eq-commutator}.
Pairing with $f$ and evaluating the eigenfunction equation at zero yield
$cI_+(f)=2\pi$ and $aI_+(f)=\bar c$; the boundary form gives
$\Ima p=|I_+(f)|^2/(4\pi)$.
For convergent potentials, a common cutoff and norm-resolvent convergence
$L^2\to Y$ give eigenfunction and boundary-trace convergence.
The differentiated equation
$(I-\widehat A_w^a)\widehat f'
=-\widehat f/(\xi+a)-i\widehat A_{xz}^a\widehat f
+\widehat w(\xi)\widehat f(0+)/(2\pi(\xi+a))$
then gives graph convergence. For resolvent images replace $a$ by
$\lambda_0$ and add $\widehat\varphi'/(\xi+\lambda_0)$.
\end{proof}

\begin{lemma}\label{rad:lem-observation}
For real $v\in\bigcap_mH^m$ there is a unitary $\Omega_{v,t}$ on $\Hp$
satisfying
\begin{equation}\label{rad:eq-covariance}
 \Omega_{v,t}\Pi v=\Pi v(t),\enspace
 L_{v(t)}\Omega_{v,t}=\Omega_{v,t}L_v.
\end{equation}
If $P_t$ is the closure of $X^*-2tL_v$ on $\Dom X^*\cap H_+^2$, then
\begin{equation}\label{rad:eq-observation}
 (\Omega_{v,t}f)(z)=\frac1{2\pi i}I_+((P_t-z)^{-1}f),\enspace z\in\C_+.
\end{equation}
\end{lemma}
\begin{proof}
Let $U'=i(T_{|D|v(t)}-T_{v(t)}^2)U$, $U(0)=I$, and
$V=Ue^{itL_v^2}$. The smooth Lax equations \cite[Section~2]{GG} give
\eqref{rad:eq-covariance}. Its generator has Fourier expression
$\widehat{C_tf}(\xi)=i\xi^2\widehat f(\xi)
-(i/\pi)\int_0^\infty\min(\xi,\eta)\widehat v(t,\xi-\eta)
\widehat f(\eta)\,d\eta$.
For $(S_hf)^{\wedge}(\xi)=\widehat f(\xi+h)$, subtraction gives
$h^{-1}[S_h,C_t]f=i(2D+h)S_hf-2iS_hT_{v(t)}f+R_{h,t}f$,
where $\|R_{h,t}f\|_2\lesssim\sqrt h\|v(t)\|_2\|f\|_2$.
Duhamel's formula yields $X^*V=V(X^*-2tL_v)$ on the core.
Frequency cutoffs and the bounded perturbation $2tT_v$ identify the
maximal dissipative closure as $P_t=V^*X^*V$.
Since $C_t$ is $O(\xi)$ at zero, $I_+Vf=I_+f$; graph continuity and
Hardy evaluation prove \eqref{rad:eq-observation} with $\Omega_{v,t}=V$.
\end{proof}

\subsection{The continuous spectral transform and moving resolvent}
\begin{lemma}\label{rad:lem-static}
Let $v\in H^\rho(\R;\R)\cap L^1(\R)$, $-\frac12<\rho<0$.
For every $E>0$ there are unique bounded solutions
$$
 m_-=e^{iEx}+G_E^+(vm_-),\enspace
 m_+=e^{iEx}+G_E^-(vm_+),\enspace
 G_E^\pm=(D-E\mp i0)^{-1}\Pi.
$$
The maps $\mathcal F_v^\pm\Pi g(E)=\int g\overline{m_\pm(\cdot,E)}$,
initially defined for $g\in C_c^\infty$, extend to contractions
$\Hp\to\mathfrak h_+$, vanish on the negative spectral subspace,
and are unitary from $E_{L_v}((0,\infty))\Hp$ onto $\mathfrak h_+$,
intertwining $L_v$ with multiplication by $E$.
If real $v_n\to v$ in $H^\rho\cap L^1$, then, for every
$\chi\in C_c((0,\infty))$, the operators
$\chi\mathcal F_{v_n}^\pm$ and $(\mathcal F_{v_n}^\pm)^*\bar\chi$
converge strongly to $\chi\mathcal F_v^\pm$ and
$(\mathcal F_v^\pm)^*\bar\chi$, respectively. On every compact
$J\Subset(0,\infty)$ the Jost functions are uniformly bounded,
converge locally in $H^r$ for $-\rho<r<\frac12$, and their sources
$vm_\pm$ converge in $L^1$, uniformly in $E\in J$.
If also $v\in H^s$, $s>-\frac12$, set
$\mathcal R_\pm(E)=\int v\overline{m_\pm(\cdot,E)}$. Then
\begin{equation}\label{rad:eq-source-transform}
 \mathcal F_v^\pm[\psi(L_v)\Pi v]=\psi\mathcal R_\pm,
 \enspace \psi\in C_c((0,\infty)),
\end{equation}
and the measure in Lemma~\ref{rad:lem-source} has positive part
$(E+\lambda_0)^{2s}|\mathcal R_\pm(E)|^2\,dE/(2\pi)$.
\end{lemma}
\begin{proof}
The free kernels satisfy
$$
 G_E^+(x)=ie^{iEx}\mathbf1_{x>0}+K_E(x),\enspace
 G_E^-(x)=-ie^{iEx}\mathbf1_{x<0}+K_E(x),
$$
where $K_E$ has multiplier $-\mathbf1_{\xi<0}/(\xi-E)$ and
$|K_E(x)|\lesssim_J1+\log_+(1/|x|)$, $|xK_E(x)|\lesssim_J1$.
Let $m=G_E^\pm(vm)$ be bounded.
Since $vm\in L^1$, the free decomposition gives
$m\in H^r_{\rm uloc}$ for every $r<\frac12$.
Choose $-\rho<r<\frac12$. Localized multiplication first gives
$vm\in H^{\rho+r-1/2-\varepsilon}$ with exponent greater than
$-\frac12$, and hence $m\in H^{1/2}_{\rm uloc}$.
A second product gives $vm\in H^{\rho-\delta_0}$, where
$0<\delta_0<\min\{\rho+\frac12,1+2\rho\}$.
Thus $m\in H^{1+\rho-\delta_0}_{\rm uloc}$ multiplies $H^\rho$
boundedly. Another use of the equation yields
$p=e^{-iEx}vm\in H^\rho\cap L^1$ and $h=e^{-iEx}m\in\mathcal X_0:=H^{1+\rho}_{\rm uloc}$.
Put $k_E=e^{-iEx}K_E$. Its real multiplier has order minus one,
so $\int\overline p(k_E*p)$ is real; the pairing is justified by
$1+2\rho>0$. For the plus sign,
$H=h-k_E*p=i\int_{-\infty}^xp$ satisfies
$(|H|^2)'=-2\operatorname{Im}(\overline p(k_E*p))$ because $v$ is real.
Integration gives $H(+\infty)=0$ and $\int p=0$; the minus sign
gives the same conclusion.

Choose smooth $\zeta_-+\zeta_+=1$, equal to one on the respective
left and right exteriors, and put
$$
 Ap=i\zeta_-\int_{-\infty}^xp-i\zeta_+\int_x^\infty p,
 \enspace C_Ep=Ap+k_E*p,
 \enspace\mathcal X_1=\{h\in\mathcal X_0:xh\in\mathcal X_0\}.
$$
The zero integral gives $h=C_E(vh)$. Write $v=v_0+v_1$ with
$v_0\in C_c^\infty$ and $\|v_1\|_{H^\rho}+\|v_1\|_1$ small.
Sobolev multiplication makes $C_EM_{v_1}$ a contraction on
$\mathcal X_0$. It is also a contraction on $\mathcal X_1$:
the outward integrals satisfy
$\|xA(v_1h)\|_\infty\lesssim\|v_1\|_1\|h\|_{\mathcal X_1}$
by $|x/y|\le1$ on each exterior; differentiation gives the
$H^{1+\rho}_{\rm uloc}$ bound. For the other term use
$$
 x(k_E*(v_1h))=k_E*(v_1xh)+(xk_E)*(v_1h),\enspace
 i\partial_\xi\!\left(-\frac{\mathbf1_{\xi<-E}}\xi\right)
 =-\frac iE\delta_{-E}+i\frac{\mathbf1_{\xi<-E}}{\xi^2}.
$$
The delta term is a bounded plane wave controlled by $\|v_1h\|_1$;
the remaining multiplier gains two derivatives. The compact source
$v_0h$ and its first moment give $C_E(v_0h)\in\mathcal X_1$.
The Neumann solution in $\mathcal X_1$ equals the given
$\mathcal X_0$ solution, so $m=O(\langle x\rangle^{-1})\in L^2$.
The equation $Dm=Em+\Pi(vm)$ then gives $m\in H^{1+\rho}_+$,
in the Lax form domain. It would be a positive eigenfunction,
which is excluded by \cite[Theorem~1.4]{BKV26}.

For the inhomogeneous equations choose
$$
 -\rho<r<\tfrac12,\enspace 0<\delta<\rho+r,\enspace
 \sigma=\rho+r-\tfrac12-\delta> -\tfrac12,\enspace
 \mathcal X=L^\infty\cap H^r_{\rm uloc},\enspace
 \mathcal Y=H^\sigma\cap L^1.
$$
Localized Sobolev multiplication gives
$\|vh\|_{\mathcal Y}\lesssim(\|v\|_{H^\rho}+\|v\|_1)\|h\|_{\mathcal X}$.
Mollification of the bounded factor identifies this Sobolev product
with the classical product. Approximation of $v$ by smooth compact
functions and local Rellich compactness make
$M_v:\mathcal X\to\mathcal Y$ compact.
On either closed half-strip adjacent to $J\Subset(0,\infty)$,
the free decomposition gives uniformly
$G_z:\mathcal Y\to L^\infty\cap H^{1+\sigma}_{\rm uloc}\to\mathcal X$.
Hence $S_z=M_vG_z$ is collectively compact on $\mathcal Y$.
For fixed sources the Volterra differences converge locally in
$H^1$, and the remaining multipliers converge in
$H^\sigma\to H^{1+\sigma}$ norm. Compact approximation of $v$
therefore gives strong continuity of $S_z$ up to the boundary.
The homogeneous argument proves injectivity of $I-S_E$;
for nonreal $z$, self-adjointness proves injectivity because
$G_zp\in H^{1+\sigma}_+$ lies in the form domain.
Fredholm theory gives the inverses. Their norms are uniform:
otherwise collective compactness would yield a unit vector in
the kernel of a limiting $I-S_z$. This argument also applies
uniformly to convergent $H^\rho\cap L^1$ families, since their
multiplication operators converge in operator norm.
The identity
$(I-G_zM_v)^{-1}=I+G_z(I-M_vG_z)^{-1}M_v$
gives the asserted continuity and bounds.

Set $q_\pm=(I-M_vG_E^\mp)^{-1}(ve^{iE\cdot})$ and
$m_\pm=e^{iE\cdot}+G_E^\mp q_\pm$. Bounded Jost candidates
belong to $\mathcal X$ by their $L^1$ sources, so uniqueness follows.
After modulation the multiplier term lies in
$H^{1+\sigma}\subset C_0$; the Volterra term gives the unit
incoming amplitude. The flux identity gives unit modulus at
the opposite endpoint.
For $g\in C_c^\infty$, put $p_z=(I-S_z)^{-1}g$.
The form identity gives
$G_zp_z=(L_v-z)^{-1}\Pi g$. Effective sources converge in
$\mathcal Y$ uniformly on $J$. Adjunction follows from absolute
Fubini for the bounded Volterra kernel and from the Sobolev
pairing $H^\sigma\times H^{1+\sigma}$ for the other term.
Consequently,
$\int e^{-iEx}p_{E\mp i0}=\int g\overline{m_\mp}$.
The free jump kernel tends boundedly to $ie^{iE(x-y)}$;
resolvent algebra and source convergence give uniformly on $J$
$$
 \lim_{\eta\to0}
 \left\langle\big[(L_v-E-i\eta)^{-1}-(L_v-E+i\eta)^{-1}\big]
 \Pi g,\Pi g\right\rangle
 =i\left|\int g\overline{m_\pm}\right|^2.
$$
Stone's formula and polarization yield
$\langle\psi(L_v)\Pi g,\Pi g\rangle
=\int_0^\infty\psi(E)|\mathcal F_v^\pm\Pi g(E)|^2\,dE/(2\pi)$.
Uniformity excludes singular positive spectrum. These identities
give the isometry and intertwining, so the closed range reduces
all scalar multiplication operators and equals $\mathfrak h(S)$
for a measurable set $S$. A countable dense family of compactly supported smooth
test functions detects every nonzero normalized Jost function outside
one null set; hence $S$ has full measure.
For $\chi\in C_c((0,\infty))$, local Jost convergence gives weak
convergence of $\chi\mathcal F_{v_n}^\pm$ on compactly supported smooth test functions
and hence on $\Hp$. The identity
$\|\chi\mathcal F_{v_n}^\pm q\|^2
=\langle|\chi|^2(L_{v_n})q,q\rangle$
and norm-resolvent convergence give convergence of norms, hence
strong convergence. The adjoints converge weakly and have norms
$\|(\mathcal F_{v_n}^\pm)^*\bar\chi h\|_2=\|\bar\chi h\|_{\mathfrak h_+}$
by surjectivity, so they also converge strongly.
Approximate $v$ in $H^s\cap H^\rho\cap L^1$ by real
smooth compact functions. Lemma~\ref{rad:lem-source} gives
convergence of the completed sources, while
$\mathcal R_{\pm,n}\to\mathcal R_\pm$ locally uniformly.
Applying the preceding convergence with $\chi=1$ on
$\operatorname{supp}\psi$ proves \eqref{rad:eq-source-transform}
and the density formula.
\end{proof}

\begin{lemma}\label{rad:lem-kernel}
Fix $J\Subset(0,\infty)$, $\gamma>0$, $t\ge1$, and set
$\epsilon=(2t)^{-1}$, $M_\epsilon f=e^{i\epsilon x^2/2}f$,
$K_{t,E}=2t(X^*-2tD+2tE-i\gamma)^{-1}\Pi$ and
$\widetilde K_{t,E}=M_\epsilon^{-1}K_{t,E}M_\epsilon$.
For $d=x-y$,
\begin{equation}\label{rad:eq-moving-kernel}
\begin{aligned}
 K_{t,E}(x,y)&=\frac i2\one_{d<0}e^{id^2/(4t)+i(E+\epsilon y)d+\gamma\epsilon d}\\
 &\enspace-\frac1{2\pi}\operatorname{p.v.}\int_0^\infty
 \frac{e^{i\eta^2/(4t)-i(E+\epsilon y)\eta-\gamma\epsilon\eta}}{\eta+d}\,d\eta,
\end{aligned}
\end{equation}
\begin{equation}\label{rad:eq-moving-bound}
 |K_{t,E}(x,y)|\le C_\gamma\left(1+
 \log_+\frac1{|x-y|(|E+\epsilon x|+\sqrt\epsilon)}\right).
\end{equation}
With fixed compact source and target cutoffs,
$\widetilde K_{t,E}\to-G_E^-$ in $\mathcal L(H^\sigma,H^{1+\sigma})$,
uniformly on $J$, for $-1/2<\sigma<0$.
Set $F_{t,E}(y)=\int_0^\infty e^{it\eta^2-i(2tE+y)\eta-\gamma\eta}\,d\eta$,
$B_{t,E}p=(i/(2\pi))\int F_{t,E}p$, $\Pi_\epsilon=M_\epsilon^{-1}\Pi M_\epsilon$.
Then
\begin{equation}\label{rad:eq-moving-odes}
\begin{aligned}
 \partial_y(M_\epsilon(y)K_{t,E}(x,y))
 &=-(iE+\gamma\epsilon)M_\epsilon(y)K_{t,E}(x,y)+iM_\epsilon(y)\Pi(x-y),\\
 (\partial_x-iE-\gamma\epsilon)\widetilde K_{t,E}p
 &=-i\Pi_\epsilon p+M_\epsilon^{-1}B_{t,E}(M_\epsilon p).
\end{aligned}
\end{equation}
Uniformly for $0<\epsilon\le1/2$,
$\|\Pi_\epsilon\|_{H^q\to H^q}\le C_q$ for $|q|<1/2$, and
$\sup_x\|M_\epsilon(y)\Pi(x-y)\|_{H^{a-1}_y}\le C_a$ for $0<a<1/2$.
There is $h_{\epsilon,E}$ with
$(\partial_x-iE-\gamma\epsilon)h_{\epsilon,E}=M_\epsilon^{-1}$,
$\|\sqrt\epsilon h_{\epsilon,E}\|_{W^{1,\infty}}\le C_J$; also
$|B_{t,E}(M_\epsilon p)|\lesssim\sqrt\epsilon\|p\|_1$.

The function $\Psi_{t,E}(y)=\sqrt t\,e^{itE^2-i\pi/4}M_\epsilon(y)F_{t,E}(y)$
is uniformly bounded in $W^{1,\infty}$ and converges locally uniformly
to $\sqrt\pi e^{-\gamma E-iEy}$, uniformly on $J$.
\end{lemma}
\begin{proof}
The Fourier resolvent
$\widehat{(X^*-2tD-z)^{-1}f}(\xi)
=i e^{-it\xi^2-iz\xi}\int_\xi^\infty e^{it\eta^2+iz\eta}\widehat f(\eta)\,d\eta$
gives \eqref{rad:eq-moving-kernel} by inversion with Abel damping.
Put $\eta+d=\sqrt t\,w$, $A=\sqrt t(E+\epsilon x)$.
Subtract the numerator at zero on $|w|<(1+|A|)^{-1}$; its constant
part gives the logarithm in \eqref{rad:eq-moving-bound}.
Van der Corput near $2A$ and integration by parts on the remaining
$|w|\asymp R$ give respectively $C/|A|$ and $C/[R(|A|+R)]$.
For the local norm limit, subtract the limiting numerator: on bounded
intervals the difference is $\eta O(t^{-1})$ with all fixed derivatives.
Using $\eta/(\eta+d)=1-d/(\eta+d)$ reduces it to smooth kernels,
$d\log|d|$ and $d_-$; Schur's test bounds first source and target
derivatives. Beyond $L$, stationary-phase splitting gives
$C/L+Ct^{-1/2}$. Interpolation proves the asserted limit.

Differentiation gives \eqref{rad:eq-moving-odes}. The kernel $M_\epsilon(y)\Pi(x-y)$
has Fourier transform bounded by a constant and the Fresnel integral
$\int_0^{b/\sqrt\epsilon}e^{-is^2/2}ds$, giving its Sobolev bound.
Conjugating $\xi+i\epsilon\partial_\xi$ by
$e^{i\xi^2/(2\epsilon)}$ and applying the weighted Hilbert transform
bound gives the projection estimate. Van der Corput and
$F_{t,E}'=\epsilon-(iE+i\epsilon y+\gamma\epsilon)F_{t,E}$ give the
trace bounds. Use
$h_{\epsilon,E}(x)=-\int_0^\infty
 e^{-(iE+\gamma\epsilon)s-i\epsilon(x+s)^2/2}\,ds$
for the inhomogeneous equation; stationary phase at
$\eta=E+\epsilon y$ gives the limit of $\Psi_{t,E}$.
\end{proof}

\begin{lemma}\label{rad:lem-moving}
Fix $J\Subset(0,\infty)$ and $\gamma>0$.
Let $-\frac12<\rho<0$, $-\rho<r<\frac12$ and
$\sigma=\rho+r-\frac12-\delta>-\frac12$, where $\delta>0$.
Set $\mathcal X=L^\infty\cap H^r_{\rm uloc}$,
$\mathcal Y=H^\sigma\cap L^1$, with sum norms, and
$\mathcal Y_R=\{p\in\mathcal Y:\operatorname{supp}p\subset[-R,R]\}$.
For $c\in H^\rho\cap L^1$, multiplication by $c$ satisfies
$$
 \|M_c\|_{\mathcal X\to\mathcal Y}
 \le C(\|c\|_{H^\rho}+\|c\|_1),\enspace
 \int_J\|\widetilde K_{t,E}M_c\|_{\mathcal X\to\mathcal X}^2dE
 \le C(\|c\|_{H^\rho}+\|c\|_1)^2
$$
uniformly for $t\ge1$. Moreover,
$$
 \|\widetilde K_{t,E}\|_{\mathcal Y_R\to\mathcal X}
 +\|G_E^-\|_{\mathcal Y\to\mathcal X}\le C_J,
 \enspace E\in J,\enspace t\ge C_J(1+R),
$$
with $C_J$ independent of $R\ge1$.
For real $u\in H^\rho\cap L^1$, all sufficiently close real
$c\in C_c^\infty$ satisfy, for sufficiently large $t$ depending on $c$,
$\sup_{E\in J}\|(I+\widetilde K_{t,E}M_c)^{-1}\|_{ \mathcal X\to\mathcal X}\le C_u$,
where $C_u$ is independent of the chosen approximation and its support.
\end{lemma}
\begin{proof}
Sobolev multiplication and localization give $\|ch\|_{H^\sigma}\le C\|c\|_{H^\rho}
\|h\|_{H^r_{\rm uloc}}$; the $L^1$ bound follows from $h\in L^\infty$.
Local mollification of $h$ with a common $L^\infty$ bound identifies
the Sobolev product with the classical product by dominated convergence.
Write $a_E(x)=|E+\epsilon x|+\sqrt\epsilon$ and
$H_E(x)=\log_+(1/a_E(x))$. If $|x-y|\ge1$ and
$|x-y|a_E(x)<1$, then
$a_E(y)\le a_E(x)+\epsilon/a_E(x)\le2a_E(x)$.
Thus the far part of \eqref{rad:eq-moving-bound} is bounded by
$C(1+H_E(y))$. Its contribution to the $L^\infty$ operator norm
is bounded by $C\|c\|_1+C\int H_E|c|$.

For a fixed smooth cutoff $\eta(x-y)$ supported on $|x-y|\le2$,
the source identity \eqref{rad:eq-moving-odes}, the Hardy row
bound and the local logarithmic $L^2$ bound give
$\|\eta(x-\cdot)M_\epsilon K_{t,E}(x,\cdot)\|_{H^{-\sigma}} \le C(1+H_E(x))$.
Take enlarged cutoffs $\widetilde\chi_j$ and
$A_j=\|\widetilde\chi_jc\|_{H^\rho}$. Sobolev duality bounds the
near contribution by $C\sup_j A_j(1+H_E(j))\|h\|_{\mathcal X}$.
Here $H_E(x)\le\log2+H_E(j)$ for $|x-j|\le1$,
$\sum_jA_j^2\le C\|c\|_{H^\rho}^2$, and
$\sup_y\int_JH_E(y)^2dE\le C$. Minkowski and
$\sup_j b_j^2\le\sum_jb_j^2$ prove the averaged $L^\infty$ bound.

For $p=ch$, subtract $B_{t,E}(M_\epsilon p)h_{\epsilon,E}$ from
$w=\widetilde K_{t,E}p$. This term has
$L^\infty\cap H^1_{\rm uloc}$ norm at most $C\|p\|_1$.
The remaining target equation and $r-1<\sigma$ give
$\|w\|_{H^r_{\rm uloc}}\le C(\|w\|_\infty+\|p\|_{H^\sigma}+\|p\|_1)$.
Consequently a measurable majorant of
$\|\widetilde K_{t,E}M_c\|_{\mathcal X\to\mathcal X}$ is
\begin{equation}\label{rad:eq-moving-majorant}
 Q_{c,t}(E)=C\left[\|c\|_{H^\rho}+\|c\|_1+
 \int H_E(y)|c(y)|\,dy+\sup_j A_j(1+H_E(j))\right],
\end{equation}
with $\|Q_{c,t}\|_{L^2(J)}\le C(\|c\|_{H^\rho}+\|c\|_1)$.
For a source supported in $[-R,R]$ and $t\ge C_J(1+R)$,
$|E+\epsilon y|$ is bounded below on its two-unit neighborhood.
The same duality estimate near the diagonal and bound away from it give
$\|\widetilde K_{t,E}p\|_{\mathcal X}\le C_J\|p\|_{\mathcal Y}$
with a constant independent of $R$. The static bound follows
from its bounded Volterra kernel on $L^1$ and its order $-1$
multiplier on $H^\sigma$.

Lemma~\ref{rad:lem-static} gives a uniform stationary inverse for $u$.
The product and static bounds give the same inverse bound for every
sufficiently close $c$. Put $B=M_c:\mathcal X\to\mathcal Y_R$,
$A_E=-G_E^-|_{\mathcal Y_R}$ and
$A_{t,E}=\widetilde K_{t,E}|_{\mathcal Y_R}$.
The identity $(I+BA_E)^{-1}=I-B(I+A_EB)^{-1}A_E$ bounds the
reduced inverse independently of $R$ and $c$.
The local kernel limit gives $B(A_{t,E}-A_E)\to0$ in
$\mathcal L(\mathcal Y_R)$, uniformly in $E$ for each fixed $c$.
The reduced moving inverse is uniformly bounded. The identity
$(I+A_{t,E}B)^{-1}=I-A_{t,E}(I+BA_{t,E})^{-1}B$
and the estimate for sources supported in $[-R,R]$ give the bound on $\mathcal X$.
\end{proof}

\begin{lemma}\label{rad:lem-window}
Let $u\in H^s(\R;\R)\cap L^1$, $s>-\frac12$, and let
$u_n\in C_c^\infty(\R;\R)$ converge to $u$ in $H^s\cap L^1$.
For $q\in\Hp$, $J\Subset(0,\infty)$, $\gamma>0$ and
$\varphi\in\mathfrak h(J)$,
\begin{equation}\label{rad:eq-window}
 \lim_{t\to+\infty}\limsup_{n\to\infty}
 \left|\left\langle
 \sqrt{4\pi t}\,e^{itE^2-i\pi/4}
 (\Omega_{u_n,t}q)(-2tE+i\gamma)
 -e^{-\gamma E}\mathcal F_u^-q,\varphi
 \right\rangle_{\mathfrak h(J)}\right|=0.
\end{equation}
Time reflection gives the past limit with $\mathcal F_u^+$.
\end{lemma}
\begin{proof}
Choose $-\frac12<\rho<\min(s,0)$ and $r,\sigma,\mathcal X,\mathcal Y$
as in Lemma~\ref{rad:lem-moving}. First take $q=\Pi g$ with
$g\in C_c^\infty$. Given small $\eta>0$, choose real
$c\in C_c^\infty$ with $\|u-c\|_{H^\rho}+\|u-c\|_1<\eta$.
Choose $R$ containing the supports of $c,g$, and write
$A_{t,E}=\widetilde K_{t,E}$, $g_t=M_\epsilon^{-1}g$.
For the potential $c$, set
$h_{c,t,E}=(I+A_{t,E}M_c)^{-1}A_{t,E}g_t, \enspace p_{c,t,E}=g_t-ch_{c,t,E}$.
Lemma~\ref{rad:lem-moving} bounds them in $\mathcal X$ and
$\mathcal Y_R$ by constants independent of the selected $c$.
The reduced equation and local kernel convergence give
$p_{c,t,E}\to p_{c,E}$ in $\mathcal Y$, uniformly on $J$, where
$p_{c,E}=g-c(I-G_E^-M_c)^{-1}(-G_E^-g)$.
The stationary inverse identity also gives
$\sup_E\|p_{c,E}-p_{u,E}\|_{\mathcal Y}\le C_g\eta$.

For large $n$, set $d_n=u_n-c$, so
$\|d_n\|_{H^\rho}+\|d_n\|_1\le2\eta$. Let $L$ bound $(I+A_{t,E}M_c)^{-1}$ and define
$\mathcal B_{n,t}=\{E\in J:Q_{d_n,t}(E)>(2L)^{-1}\}$,
$\mathcal G_{n,t}=J\setminus\mathcal B_{n,t}$, using
\eqref{rad:eq-moving-majorant}. Then $|\mathcal B_{n,t}|\le C\eta^2$
and $\|(I+A_{t,E}M_{u_n})^{-1}\|\le2L$ on $\mathcal G_{n,t}$.
The field
$h_{n,t,E}=M_\epsilon^{-1}2t (X^*-2tL_{u_n}+2tE-i\gamma)^{-1}\Pi g$
belongs to $\mathcal X$: its effective source is compactly supported
$L^2$, and for fixed $t$ the kernel bound has a finite logarithmic
coefficient. It satisfies
$(I+A_{t,E}M_{u_n})h_{n,t,E}=A_{t,E}g_t$.
Subtracting the equation for $h_{c,t,E}$ and using the averaged estimate yields
$$
 \int_{\mathcal G_{n,t}}
 \left(\|h_{n,t,E}-h_{c,t,E}\|_{\mathcal X}^2+
       \|p_{n,t,E}-p_{c,t,E}\|_{\mathcal Y}^2\right)dE
 \le C_g\eta^2,
 \enspace p_{n,t,E}=g_t-u_nh_{n,t,E}.
$$
All constants are independent of $c,n,t$ in these ranges.

Lemma~\ref{rad:lem-observation} and the exact free resolvent give
$$
 \sqrt{4\pi t}\,e^{itE^2-i\pi/4}
 (\Omega_{u_n,t}\Pi g)(-2tE+i\gamma)
 =\frac1{\sqrt\pi}\int\Psi_{t,E}(y)p_{n,t,E}(y)\,dy.
$$
The bounded trace kernel, its locally uniform limit, and convergence
of the compact source imply that this expression differs on
$\mathcal G_{n,t}$, in $L^2(J)$, by at most $C_g\eta+o_t(1)$ from
$e^{-\gamma E}\int e^{-iEy}p_{u,E}(y)\,dy
=e^{-\gamma E}\mathcal F_u^-q$; the last identity is stationary
resolvent adjunction. On $\mathcal B_{n,t}$, each term paired with
$\varphi$ is bounded by
$\|q\|_2\|\mathbf1_{\mathcal B_{n,t}}\varphi\|_{\mathfrak h(J)}$.
For the actual term this follows from the unitary observation
and Hardy evaluation; for the stationary term it follows from
Lemma~\ref{rad:lem-static}. Absolute continuity and
$|\mathcal B_{n,t}|\le C\eta^2$ make these bounds vanish as
$\eta\to0$. Taking $n\to\infty$, then $t\to\infty$, then
$\eta\to0$ proves \eqref{rad:eq-window} for $\Pi g$.
The same contraction bounds give the result for every $q\in\Hp$
by density. Reflection gives the past limit.
\end{proof}

\begin{lemma}\label{rad:lem-moments}
If real $v\in H^s$, $s>-1/2$, has $\widehat v\in L^\infty(-1,1)$,
then its negative eigenvalues $-a_j$ satisfy
$\sum_ja_j^\gamma<\infty$ for every $\gamma>0$.
\end{lemma}
\begin{proof}
For $B_a=(D+a)^{-1/2}T_v(D+a)^{-1/2}$, integration along Fourier
kernel diagonals gives
$$
 \|B_a\|_{\Sch_2}^2=\frac1{4\pi^2}\int_\R|\widehat v(k)|^2
 \frac{\log(1+|k|/a)}{|k|}\,dk
 \le C_v(1+\log^2(1/a)),\enspace0<a\le1.
$$
Here the low-frequency estimate uses the hypothesis; on $|k|>1$ the
weight is bounded by $C_s(1+\log(1/a))\langle k\rangle^{2s}$.
Birman--Schwinger and min--max give
$N(a)=\#\{j:a_j>a\}\le\|B_a\|_{\Sch_2}^2$.
Semiboundedness and
$\sum_ja_j^\gamma=\gamma\int_0^{a_*}a^{\gamma-1}N(a)\,da$
finish the proof.
\end{proof}

\subsection{Proof of Theorem~\ref{thm:main-resolution}: soliton resolution with infinitely many solitons and radiation}
\begin{proof}
Choose real $u_n\in C_c^\infty$ converging to $u_0$ in
$H^s\cap L^1$, and write $L_t=L_{u(t)}$, $q(t)=\Pi u(t)$.
Lemma~\ref{rad:lem-source} defines the spectral projections on
$H^s$ and gives $q=q_-+q_+$, where
$q_\pm(t)=\one_{\R_\pm}(L_t)q(t)$. Its conserved measure satisfies
\begin{equation}\label{rad:eq-atom-mass}
 \nu_s\{-a_j\}=2\pi a_j(\lambda_0-a_j)^{2s}.
\end{equation}
This follows from Lemma~\ref{rad:lem-coordinate}; the pairings
are defined by choosing $-\frac12<\rho<\min(s,0)$.
Wu's identity detects every negative eigenfunction, so conservation
of this measure preserves the complete negative spectrum.

Fix $\psi\in C_c((0,\infty))$, set
$f=\psi(L_0)q(0)$ and $v_\psi(t)=\psi(L_t)q(t)$.
The vectors $f_n=\psi(L_{u_n})\Pi u_n$ converge to $f$
in $L^2$, and their observations converge to $v_\psi(t)$ in
$L^2$ at each fixed time. Lemma~\ref{rad:lem-window} and the
observation contraction thus give, for $\sigma\in\{-1,1\}$,
$$
 \sqrt{4\pi|t|}\,e^{itE^2-i\sigma\pi/4}
 v_\psi(t,-2tE+i\gamma)
 \rightharpoonup e^{-\gamma E}\psi(E)\mathcal R_{-\sigma}(E)
 \enspace(t\to\sigma\infty)
$$
in $\mathfrak h(J)$ for every $J\Subset(0,\infty)$ and $\gamma>0$.
Define the unitary rescaling from $L^2(\R)$ to $\mathfrak h(\R)$ by
$(\mathcal M_tv)(E)=\sqrt{4\pi|t|}\,e^{itE^2-i\sigma\pi/4}v(-2tE)$.
Fourier inversion gives $\mathcal M_te^{itD^2}\mathcal F^{-1}=C_t$,
where $C_t$ has Fourier multiplier $e^{-i\zeta^2/(4t)}$ in $E$;
hence $C_t,C_t^*\to I$ strongly. For
$F_t=\mathcal F e^{-itD^2}v_\psi(t)$, extended by zero to $E<0$,
one has
$\mathcal M_t(e^{-\gamma D}v_\psi(t))=C_t(e^{-\gamma E}F_t)$.
Testing against $e^{\gamma E}\varphi$, with
$\varphi\in C_c^\infty((0,\infty))$, yields
$\mathcal F e^{-itD^2}v_\psi(t) \rightharpoonup\psi\mathcal R_{-\sigma} \enspace\hbox{in }\mathfrak h_+$.
The conserved source measure and Lemma~\ref{rad:lem-static} give
$\|v_\psi(t)\|_2=\|f\|_2
=\|\psi\mathcal R_{-\sigma}\|_{\mathfrak h_+}$, so convergence
is strong. For $s>0$, interpolate with the uniform
$H^{s+1}$ bound in Lemma~\ref{rad:lem-source} and the
compact Fourier support of the free comparison; for
$s\le0$, use $L^2\subset H^s$.
The same lemma and the density
$d\nu_s(E)=(E+\lambda_0)^{2s}|\mathcal R_\pm(E)|^2 \frac{dE}{2\pi}$, $E>0$,
remove the compact spectral cutoff. Define real $u_\infty^\pm$
by $\widehat u_\infty^+(E)=\mathcal R_-(E)$,
$\widehat u_\infty^-(E)=\mathcal R_+(E)$ for $E>0$ and
$\widehat u_\infty^\pm(-E)=\overline{\widehat u_\infty^\pm(E)}$.
Then $u_\infty^\pm\in H^s$ and
\begin{equation}\label{rad:eq-continuous-resolution}
 \|2\Rea q_+(t)-e^{t\partial_x|D_x|}u_\infty^\pm\|_{H^s}\to0.
\end{equation}

For the discrete part, let $-a_{j,n}$ and $\varphi_{j,n}$ approximate
$-a_j$ and $\varphi_j$ with consistent phases. Write
$c_{j,n}=\langle\Pi u_n,\varphi_{j,n}\rangle$,
$b_{j,n}=\langle X^*\varphi_{j,n},\varphi_{j,n}\rangle$ and
$\Omega_{n,t}=\Omega_{u_n,t}$.
We use the following consequence of
Lemmas~\ref{rad:lem-coordinate} and~\ref{rad:lem-observation}:
if $v_n\to v$ in $L^2$, $\lambda_n\to\lambda<0$, and
$v\perp\Ker(L_0-\lambda)$, then
\begin{equation}\label{rad:eq-ray-decay}
 \lim_{|t|\to\infty}\limsup_n
 |(\Omega_{n,t}v_n)(z-2t\lambda_n)|=0,\enspace z\in\C_+.
\end{equation}
Indeed, the vectors $(L_0-\lambda)g$ with
$g=(L_0+\lambda_0)^{-1}\varphi$ and
$\widehat\varphi\in C_c^\infty((0,\infty))$ are dense in that
orthogonal complement. Put $g_n=(L_{u_n}+\lambda_0)^{-1}\varphi$.
The observation formula gives
$$
 [\Omega_{n,t}(L_{u_n}-\lambda_n)g_n](z-2t\lambda_n)
 =-\frac{I_+g_n}{4\pi it}
 +\frac{[\Omega_{n,t}(X^*-z)g_n](z-2t\lambda_n)}{2t}.
$$
Graph convergence bounds this by $C_{g,z}/|t|$ uniformly in $n$;
the observation bound controls the approximation error by
$C_z\varepsilon$. This proves \eqref{rad:eq-ray-decay}.
The corresponding atom identity is
$$
 [\Omega_{n,t}\varphi_{j,n}](z+2ta_{j,n})
 =\frac{I_+\varphi_{j,n}}{2\pi i(b_{j,n}-z)}
 -\left[\Omega_{n,t}\frac{(X^*-b_{j,n})\varphi_{j,n}}
 {b_{j,n}-z}\right](z+2ta_{j,n}).
$$
The last input converges to a vector orthogonal to $\varphi_j$.
Multiply by $c_{j,n}$, use $c_{j,n}I_+\varphi_{j,n}=2\pi$, and
apply \eqref{rad:eq-ray-decay}. Fixed-time spectral convergence gives
$[\one_{\{-a_j\}}(L_t)q(t)](z+2a_jt) \longrightarrow\frac{i}{z+p_j}=\Pi R_{p_j}(z)$.
The boundary form gives $\Ima p_j=|I_+\varphi_j|^2/(4\pi)=(2a_j)^{-1}$.
The atom and its limit both have squared $L^2$ norm $2\pi a_j$.
Hardy evaluation kernels are total, so convergence is strong in
$L^2$. For $s>0$, interpolate using \eqref{rad:eq-gain};
for $s\le0$, use $L^2\subset H^s$.

Lemma~\ref{rad:lem-moments} applies because $u_0\in L^1$. Since
$\widehat{R_{p_j}}(\xi)=2\pi e^{i\Rea p_j\xi-|\xi|/(2a_j)}, \enspace \|R_{p_j}\|_{H^r}\le C_{r,a_*}\sqrt{a_j}$,
the complete soliton sum converges absolutely in every $H^r$,
with tails uniform in time. The discrete source tails are
uniformly small by \eqref{rad:eq-atom-mass} and
Lemma~\ref{rad:lem-source}. Sum the fixed-atom limits, remove the
negative cutoff, and use \eqref{rad:eq-continuous-resolution}.
Unitarity of the
free group proves uniqueness; Proposition~\ref{rad:prop-reflectionless}
below gives the nonvanishing assertion.
\end{proof}

\begin{prop}\label{rad:prop-reflectionless}
For real $u_0\in H^s\cap L^1$, $s>-1/2$, either radiation profile
vanishes if and only if $u_0$ is a finite sum of $R_p$, including zero.
The number of negative eigenvalues then equals the number of profiles,
with multiplicity.
\end{prop}
\begin{proof}
Vanishing makes the positive part of $\nu_s$ zero. There is no zero
atom, and $1\le\lambda_0-a_j\le\lambda_0$, so
\eqref{rad:eq-atom-mass} gives $\sum_ja_j<\infty$.
For $c_j=\langle\Pi u_0,\varphi_j\rangle$, the series
$q=\sum_jc_j\varphi_j$ converges in $L^2$ and has the same
$(L_{u_0}+\lambda_0)^s$ image as $\Pi u_0$.
The isomorphism in \eqref{rad:eq-scales} gives $q=\Pi u_0$, hence $u_0\in L^2$
and trace equality. Corollary~\ref{cor:trace-saturation-classification}
applies; $L^1$ excludes an infinite nonnegative pole sum because every
profile has integral $2\pi$. Conversely, finite multisolitons saturate
the trace inequality, so both positive source measures vanish.
The eigenvalue count follows from the same corollary.
\end{proof}

\begin{prop}\label{rad:prop-example}
Let $A>\pi/2$, $T\ge1$, and let smooth $0\le\theta\le1$ equal zero
on $(-\infty,1]$ and one on $[2,\infty)$.
The potential $V_T(x)=A\theta((\log x)/T)/(x(\log x)^2)$ for $x>1$,
extended by zero, belongs to $L^1\cap\bigcap_{m\ge0}H^m$ and has
infinitely many negative eigenvalues. Its solution has infinitely many
solitons and nonzero radiation in both directions.
\end{prop}
\begin{proof}
Choose $h\in C_c^\infty((1,B);\R)$ with
$\alpha=\int|h'|^2<\beta=(A/(2\pi))\int|h(r)|^2r^{-2}dr$;
$h(r)=\sqrt r\,\chi((\log r)/L)$, with fixed nonzero
$\chi\in C_c^\infty((0,1))$, has quotient tending to $1/4<A/(2\pi)$ as $L\to\infty$.
Take $T_{j+1}>BT_j$, $T_j\ge2T$, and put
$G_j(\tau)=\sqrt{T_j}h(\tau/T_j)$, $g_j=-G_j'$,
$\widehat f_j(\xi)=\sqrt{2\pi}\one_{\xi>0}\xi^{-1}g_j(\log(1/\xi))$.
Disjoint Fourier supports give kinetic form $\alpha\sum|z_j|^2$.
Set $K(\sigma)=e^{ie^\sigma}-\one_{\sigma<0}$.
Its primitive is bounded: the negative half-line is absolutely integrable,
and $r=e^\sigma$ reduces the positive half-line to $\int_1^R e^{ir}dr/r$.
Fourier inversion and integration by parts therefore give, uniformly in $t$,
$\sqrt{2\pi}f_j(e^t)=G_j(t)+O(\|g_j'\|_1)
=G_j(t)+O(T_j^{-1/2})$.
The disjoint functions $M_j(x)=(2\pi)^{-1/2}G_j(\log x)\one_{x>0}$
have potential form $\beta\sum|z_j|^2$, while the error map
$(z_j)\mapsto\sum z_j(f_j-M_j)$ into $L^2(V_Tdx)$ has norm
$C\|V_T\|_1^{1/2}(\sum T_j^{-1})^{1/2}$.
Choose it below $\sqrt\beta-\sqrt\alpha$. Every nonzero finite
combination has negative Lax form, so min--max gives infinitely many
eigenvalues. Apply Theorem~\ref{thm:main-resolution} and
Proposition~\ref{rad:prop-reflectionless}.
\end{proof}

For $-\frac12<s<0$, choose nonnegative $\chi\in C_c^\infty(\R)$
equal to one near zero and set $\alpha=\frac12-\frac s2$.
Then $v=-\chi|x|^{-\alpha}$ belongs to $H^s\cap L^1\setminus L^2$:
$\widehat v(\xi)=O(|\xi|^{\alpha-1})$ at infinity and
$s+\alpha<\frac12<\alpha<1$. Its Lax form is nonnegative,
so Proposition~\ref{rad:prop-reflectionless} gives nonzero radiation
in both directions and no solitons.

\section{Explicit remainder estimates and decay rates}
\label{sec:quantitative}\label{subsec:proof-main-modulus}
Fix $u_0=u[p^{\rm in}]$, $p^{\rm in}\in\mathfrak P_{\rm in}$, set
$L=L_{u_0}$, $q=\Pi u_0$, and use its spectral data
$\lambda_j=-a_j$, $\varphi_j$, $p_j^\infty$.
Choose $\gamma_j=\langle q,\varphi_j\rangle>0$, and set
$m_j=\langle X^*\varphi_j,\varphi_j\rangle$, $z_j=\overline{m_j}$,
$q_J=\sum_{k>J}a_k$ and $P_J=\operatorname{proj}_{\spn\{\varphi_k:k\le J\}}$.
Wu's identity and Lemma~\ref{lem:Xstar-discrete} give
$\gamma_j^2=2\pi a_j$, $I_+(\varphi_j)=\gamma_j/a_j$, $q=\sum_j\gamma_j\varphi_j$, $X^*\varphi_j\in\Dom(L)\cap\Hd$, $\Ima z_j=(2a_j)^{-1}$.
Thus $f_j=q-\gamma_j\varphi_j$ and
$g_j=(X^*\varphi_j-m_j\varphi_j)/(m_j-z_j)$ belong to
$\Hd\cap\varphi_j^\perp$, where $m_j-z_j=-\ii/a_j$.
For $k\ne j$, $z\in\C_+$ and $J\ge N$, define
$$
\begin{aligned}
 \kappa_{jk}(z)&=\frac{|I_+(\varphi_k)|/(4\pi)
 +\|X^*\varphi_k-z\varphi_k\|_2/(4\sqrt{\pi\Ima z})}
 {|\lambda_k-\lambda_j|},\\
 \Theta_{N,J}&=\sqrt\pi\sum_{j\le N}(\Ima z_j)^{-1/2}
 \bigl((2\pi q_J)^{1/2}+\gamma_j\|(I-P_J)g_j\|_2\bigr),\\
 \Xi_{N,J}&=2\pi\sum_{j\le N}\sum_{\substack{k\le J\\k\ne j}}
 (\gamma_k+\gamma_j|\langle g_j,\varphi_k\rangle|)\kappa_{jk}(z_j),
 \enspace K_N^\to=\min\{J\ge N:\Theta_{N,J}\le2\pi q_N\},\\
 \mathcal T_N^\Omega&=\max_{j<k\le N}
 \frac{|\Rea(p_j^\infty-p_k^\infty)|}{|\lambda_j-\lambda_k|},\enspace
 \mathcal B_N^\Omega=2\Xi_{N,K_N^\to}
 +4\pi\sum_{j<k\le N}|\lambda_j-\lambda_k|^{-1}.
\end{aligned}
$$
An empty maximum is zero. These constants are finite:
$\Theta_{N,J}\to0$, $q_N>0$, and the eigenvalues are distinct.
Write $\usol(t)=\sum_jR_{p_j^\infty}(\cdot+2t\lambda_j)$.
For $s>\frac12$, put $\vartheta_s=\frac12-\frac1{4s}$,
$M_s=\sup_t\|u(t)\|_{\dot H^s}$ and
$\mathcal M_s=M_s+2\sqrt{\pi\Gamma(2s+1)}a_1^{s-1/2}\sum_ja_j, \enspace C_{s,u_0}=2^{\vartheta_s}C_s^{\rm GN}\mathcal M_s^{1/(2s)}$,
where $C_s^{\rm GN}$ is the Gagliardo--Nirenberg constant.

\begin{prop}\label{thm:main-modulus}
If $M_s<\infty$ for some $s>\frac12$, then, for every $N\ge1$ and
$0<|t|\ge\mathcal T_N^\Omega$,
\begin{equation}\label{eq:main-bound}
 \|u(t)-\usol(t)\|_\infty
 \le4q_N+C_{s,u_0}
 \left(6\pi q_N+\frac{\mathcal B_N^\Omega}{|t|}\right)^{\vartheta_s}.
\end{equation}
The infimum $\omega_s^\Omega(t)$ of the right-hand side over admissible
$N$ tends to zero as $|t|\to\infty$.
\end{prop}

\begin{proof}
Set $\Omega_tF(z)=(2\ii\pi)^{-1}I_+((X^*-2tL-z)^{-1}F)$.
By \cite[Lemma~5.1]{GG},
$|\Omega_tF(z)|\le\frac{\|F\|_2}{2\sqrt{\pi\Ima z}}$.
Applying the resolvent to
$(X^*-2tL-z+2t\lambda_j)\varphi_k
=(X^*-z)\varphi_k\allowbreak-2t(\lambda_k-\lambda_j)\varphi_k$
therefore gives, for $v\in\Hd\cap\varphi_j^\perp$,
$$
 |\Omega_tv(z-2t\lambda_j)|
 \le |t|^{-1}\sum_{\substack{k\le J\\k\ne j}}
 |\langle v,\varphi_k\rangle|\kappa_{jk}(z)
 +\frac{\|(I-P_J)v\|_2}{2\sqrt{\pi\Ima z}}.
$$
Let $\mathcal K_j=\Pi R_{p_j^\infty}=\ii/(\cdot-m_j)$ and
$\mathcal K_{j,t}=\mathcal K_j(\cdot+2t\lambda_j)$.
The explicit formula and the Hardy reproducing identity
\cite[proof of Lemma~5.3]{GG} yield
$$
\begin{aligned}
 \langle\Pi u(t),\mathcal K_{j,t}\rangle-2\pi a_j
 &=2\pi\bigl(\Omega_tf_j-\gamma_j\Omega_tg_j\bigr)(z_j-2t\lambda_j),\\
 \langle\mathcal K_{j,t},\mathcal K_{k,t}\rangle
 &=\frac{2\pi\ii}{2t(\lambda_j-\lambda_k)
 +p_j^\infty-\overline{p_k^\infty}}.
\end{aligned}
$$
Since $\|(I-P_J)f_j\|_2^2=2\pi q_J$, the sum of the absolute
overlap defects is at most $\Theta_{N,J}+\Xi_{N,J}/|t|$.
Expand the square of $\Pi r_N$, where
$r_N=u-\sum_{j\le N}R_{p_j^\infty}(\cdot+2t\lambda_j)$,
and use $\|\Pi u(t)\|_2^2=2\pi\sum_ja_j$.
For $|t|\ge\mathcal T_N^\Omega$, each cross denominator has
modulus at least $|t||\lambda_j-\lambda_k|$; hence
$$
 \|\Pi r_N(t)\|_2^2
 \le2\pi q_N+2\Theta_{N,J}+\frac{2\Xi_{N,J}}{|t|}
 +\frac{4\pi}{|t|}\sum_{j<k\le N}|\lambda_j-\lambda_k|^{-1}.
$$
Taking $J=K_N^\to$ gives
$\|\Pi r_N(t)\|_2^2\le6\pi q_N+\mathcal B_N^\Omega/|t|$.
The identity
$\|R_{p_j^\infty}\|_{\dot H^s} =2\sqrt{\pi\Gamma(2s+1)}a_j^{s+1/2}$
gives $\|r_N(t)\|_{\dot H^s}\le\mathcal M_s$.
Interpolation, $\|r_N\|_2^2=2\|\Pi r_N\|_2^2$, and the tail bound
$\|\usol-\sum_{j\le N}R_{p_j^\infty}(\cdot+2t\lambda_j)\|_\infty
\le4q_N$ prove \eqref{eq:main-bound}. Letting $|t|\to\infty$ and
then $N\to\infty$ proves the assertion about $\omega_s^\Omega$.
\end{proof}

\begin{coro}\label{cor:general-power}
Under Proposition~\ref{thm:main-modulus}, suppose, for all sufficiently
large $N$, that $q_N\le C_0N^{-p}$,
$\mathcal T_N^\Omega\le C_0N^\tau$ and
$\mathcal B_N^\Omega\le C_0N^\nu$, where $p>0$, $\tau,\nu\ge0$.
Then $\|u(t)-\usol(t)\|_\infty=O(|t|^{-\vartheta_sp/m})$,
where $m=\max\{\tau,\nu+p\}$.
\end{coro}
\begin{proof}
Choose $N=\lfloor c|t|^{1/m}\rfloor$, with $c>0$ small enough
for $|t|\ge\mathcal T_N^\Omega$. Then
$q_N+\mathcal B_N^\Omega/|t|=O(|t|^{-p/m})$, and
\eqref{eq:main-bound} applies.
\end{proof}

\begin{coro}\label{cor:target-power}
For every $\alpha>1$ there is an infinite multisoliton with
$a_j=j^{-\alpha}$, $p_j^\infty=\ii j^\alpha/2$ and
$M_s<\infty$ for all $s>\frac12$. It satisfies
$\|u(t)-\usol(t)\|_\infty=O(|t|^{-\rho_{\alpha,s}})$, $\rho_{\alpha,s}=\frac{\vartheta_s(\alpha-1)^2}{7\alpha^2+4\alpha-4}$.
\end{coro}
\begin{proof}
Use Theorem~\ref{thm:exact-range} with $w_j=\ii j^\alpha/2$.
The inverse matrices of its finite approximants have imaginary
part at least $I/2$. Their Poisson profiles thus have widths at
least $1/2$, and $|\partial_x^ku_N|\le k!2^ku_N$.
The uniform conserved $L^2$ bound passes these derivative estimates
to the limit, giving $M_s<\infty$.

Here $q_N\asymp N^{1-\alpha}$ and $\mathcal T_N^\Omega=0$.
The commutator of Lemma~\ref{lem:Xstar-discrete} gives
$$
 |\langle g_j,\varphi_k\rangle|
 =\frac{\sqrt{a_ja_k}}{|a_j-a_k|}\ (k\ne j),\enspace
 \|X^*\varphi_k\|_2^2
 =\frac1{4a_k^2}+\sum_{\ell\ne k}
 \frac{a_\ell}{a_k(a_k-a_\ell)^2}\lesssim k^{2\alpha+2}.
$$
For the last estimate split at $k/2$ and $2k$, using
$|a_k-a_\ell|\gtrsim k^{-\alpha-1}|k-\ell|$ in the middle range.
If $a_{J+1}\le a_N/2$, then
$\|(I-P_J)g_j\|_2^2\le4q_J/a_j$ for $j\le N$, whence
$\Theta_{N,J}\lesssim N^{1/2}J^{(1-\alpha)/2}$.
Consequently $K_N^\to\lesssim N^r$ with
$r=(2\alpha-1)/(\alpha-1)$.
For $j,k\le J$, $j\ne k$, use
$|a_j-a_k|\gtrsim J^{-\alpha-1}$ to obtain
$\kappa_{jk}(z_j)\lesssim J^{2\alpha+2}$, $\gamma_k+\gamma_j|\langle g_j,\varphi_k\rangle|\lesssim J^{\alpha+1}, \enspace \Xi_{N,J}\lesssim NJ^{3\alpha+4}$.
The remaining inverse-gap sum is $O(N^{\alpha+3})$, so
$\mathcal B_N^\Omega\lesssim N^{1+r(3\alpha+4)}$.
Corollary~\ref{cor:general-power}, with $p=\alpha-1$ and $\tau=0$,
gives the stated exponent.
\end{proof}

\section{Limits of finite multisolitons and local spectral traces}
\label{gas:sec-phase-space}
Fix $B>0$. For bounded real $q$, write $H_+=L^2_+(\R)$, $T_qf=\Pi(qf)$ and
$L_q=D-T_q$ on $H_+^1$, and use
$\widehat{\mathcal Hf}=-i\operatorname{sgn}(\xi)\widehat f$.
In localized traces, spectral operators on $H_+$ are extended by zero.
The finite inverse formula of Section~\ref{sec:exact-range-proof} reads
\begin{equation}\label{gas:finite-reconstruction}
 u_N=-2\Im\Tr(x+B_N(t))^{-1},\enspace
 (B_N)_{jk}=\begin{cases}-c_j-2a_jt+i/(2a_j),&j=k,\\
 i/(a_j-a_k),&j\ne k,
 \end{cases}
\end{equation}
where $a_j>0$ are distinct and $c_j\in\R$. Its negative spectrum is
$\{-a_j\}$. Finite multisolitons are precisely finite positive Poisson
sums \cite{Sun21}, with $\int u_N=2\pi N$ and
$\int u_N^2=4\pi\sum_j a_j$. Let $\mathcal F_B$ consist of these
profiles with $a_j\le B$, including zero, and put
$\mathcal X_B=\overline{\mathcal F_B}^{\,C^\infty_{\rm loc}}$.

\begin{theorem}
\label{gas:global-flow}\label{gas:compact-flow}
The space $\mathcal X_B$ is compact, and $L_q\ge-B$ for $q\in\mathcal X_B$.
The finite multisoliton flow extends uniquely to a continuous group
$S(t)$ of global classical BO solutions on $\mathcal X_B$, commuting
with translations. For $k\ge1$ and $|\sigma|<(2B)^{-1}$,
$$
 0\le S(t)q\le4B,\enspace
 |\partial_x^kS(t)q|\le4Bk!(2B)^k,\enspace
 |S(t)q(x+i\sigma)|\le\frac{4B}{(1-2B|\sigma|)^2}.
$$
All mixed derivatives have bounds depending only on $B$ and their
orders, and local uniform convergence $q_n\to q$ implies
$S(\cdot)q_n\to S(\cdot)q$ in $C^\infty_{\rm loc}(\R^2)$.
\label{gas:trajectory-continuity}\label{gas:bounded-poisson-limit}
Every Poisson sum $q=\sum_j2y_j/((x-x_j)^2+y_j^2)\le B$ belongs
to $\mathcal X_B$; its finite truncations converge in this flow topology.
\end{theorem}
\begin{proof}
The eigenvalues $\gamma_j+iy_j$ of $B_N$ satisfy
$y_j\ge\eta=(2B)^{-1}$ and give the Poisson representation of $u_N$.
Testing $L_{u_N}\ge-B$ on $\sqrt{\eta/\pi}(s-x+i\eta)^{-1}$ gives
$P_\eta*u_N\le2B$. Each convolved summand is at least half the
original, so $u_N\le4B$. For $F_y(x)=2y/(x^2+y^2)$,
$|F_y^{(k)}|\le k!y^{-k}F_y$ and
$|F_y(x+i\sigma)|\le(1-|\sigma|/y)^{-2}F_y(x)$, proving the bounds.
For an even cutoff $\chi=1$ near zero, put
$K_R(z)=(1-\chi(z/R))/z$. Twice integrating by parts gives
$$
 \mathcal Hf''=\frac1\pi\operatorname{p.v.}\int
 \frac{\chi(z/R)}z f''(\cdot-z)\,dz+
 \frac1\pi\int K_R''(z)f(\cdot-z)\,dz,
 \enspace \|K_R''\|_1=O(R^{-2}).
$$
This proves local smooth continuity under common derivative bounds;
the equation then bounds mixed derivatives and yields subsequential limits.

For uniqueness, set $w=u-v$, $a=u+v$, $g_R=(x+iR)^{-1}$ and
$f_R=g_Rw$. Symmetric truncation gives
$\partial_tf_R-\mathcal Hf_R''=2g_R\mathcal Hf_R'-af_R'-(a'+ag_R)f_R$.
The $L^2$ energy estimate and
$\|f_R'\|_2\le\sqrt{\pi/R}\|w_x\|_\infty+R^{-1}\|f_R\|_2$ imply
$\|f_R(t)\|_2\le e^{C|t|}\|f_R(0)\|_2+ C_TR^{-3/2}\sup_{|s|\le T}\|w_x(s)\|_\infty$.
If $w(0)=0$, multiply the squared estimate by $r^2+R^2$ and send
$R\to\infty$ to obtain $w=0$ on $[-r,r]$.
Arzel\`a--Ascoli and uniqueness now give compactness, continuity and
the group law. Dominated convergence preserves $L_q\ge-B$.
A Poisson sum bounded by $B$ has $y_j\ge2/B$; its partial
sums converge locally smoothly, satisfy $L_{q_N}\ge-B$, and hence
have the asserted flow limit.
\end{proof}

\begin{lemma}\label{gas:bounded-inner-model}
If $\Theta$ is inner on $\C_+$ and $\Theta'\in H^\infty$, then
$q=-i\Theta'/\Theta$ is bounded and nonnegative on $\R$. Writing
$K_\Theta=H_+\ominus\Theta H_+$, one has
$$
 P_q:=\mathbf1_{(-\infty,0)}(L_q)=P_{K_\Theta},\enspace
 P_q(x,x)=\frac{q(x)}{2\pi},\enspace
 \Tr(\chi P_q\chi)=\frac1{2\pi}\int\chi^2q
$$
for bounded compactly supported real $\chi$. The complementary
restriction is unitarily equivalent to $D$, and $\ker L_q=0$.
\end{lemma}
\begin{proof}
Schwarz reflection and canonical factorization \cite{Garnett} give
$q\ge0$. Multiplication by $\Theta,\bar\Theta$ preserves $H^1$;
thus $[D,P_{K_\Theta}]$ is bounded and
$A=P_{K_\Theta}D|_{K_\Theta}$ is self-adjoint and nonnegative on
$K_\Theta\cap H^1$. The domain-preserving involution $Cf=\Theta\bar f$ gives
\begin{equation}\label{gas:inner-conjugation}
 L_q(\Theta f)=\Theta Df,\enspace L_q|_{K_\Theta}=-CAC.
\end{equation}
Since $\langle Af,f\rangle=\langle Df,f\rangle>0$ for $f\ne0$,
these identities prove the spectral assertions. The boundary diagonal
of the model-space kernel
$i(1-\Theta(z)\overline{\Theta(w)})/(2\pi(z-\bar w))$
\cite{GarciaModel} gives the trace formulas.
\end{proof}

\begin{theorem}
\label{gas:inner-classification}\label{gas:local-trace-continuity}
The space $\mathcal X_B$ consists exactly of the bounded inner
potentials satisfying $L_q\ge-B$. They admit locally smooth representations
$q(x)=c+\sum_j2y_j/((x-x_j)^2+y_j^2)$, where
$c\ge0$ and $y_j\ge(2B)^{-1}$.
If $q_n\to q$ in $\mathcal X_B$, $F\in C_0(\R)$, and $\chi$ is
bounded, real and compactly supported, then
\begin{equation}\label{gas:local-functional-convergence}
 \|\chi(P_{q_n}-P_q)\chi\|_{\Sch_1}+
 \|\chi F(L_{q_n})P_{q_n}\chi-\chi F(L_q)P_q\chi\|_{\Sch_1}\to0,
\end{equation}
uniformly on compact time intervals along convergent solutions.
\end{theorem}
\begin{proof}
Normalize finite inner functions by
$\Theta_n(x)=\exp(i\int_0^xq_n)$. The Poisson formula and
$\|\Theta_n'\|_\infty\le4B$ give an inner limit with bounded derivative.
Conversely, reflection excludes real singular factors, leaving an
exponential and Blaschke factors. By \eqref{gas:inner-conjugation},
$A\le B$; testing on $(x-x_j+iy_j)^{-1}\in K_\Theta$ gives the
height bound, and positivity gives local smooth convergence.
To approximate $c>0$ from below, choose $d_y\coth(d_yy/2)=c$, $y>2/c$:
$$
 \sum_{j\in\Z}\frac{2y}{(x-2\pi j/d_y)^2+y^2}
 =d_y\frac{\sinh(d_yy)}{\cosh(d_yy)-\cos(d_yx)}\le c
 \longrightarrow c.
$$
Finite diagonal truncations yield $0\le q_n\le q$, local smooth
convergence, and $L_{q_n}\ge L_q\ge-B$.
The projections $P_{q_n}=\Pi-M_{\Theta_n}\Pi M_{\bar\Theta_n}\to P_q$ strongly.
The local trace formula gives convergence of the Hilbert--Schmidt
norms of $P_{q_n}\chi$, hence convergence in that norm. Strong
resolvent convergence $L_{q_n}\to L_q$ and factorization through
$P_{q_n}\chi$ prove \eqref{gas:local-functional-convergence}; compactness
gives uniformity in time.
\end{proof}

\section{Soliton gases with prescribed spectral bands}
\label{gas:arbitrary-band-section}
Fix $p_1<q_1<\cdots<p_r<q_r<0$, and put
$K_j=q_j-p_j$, $\Omega_j=-K_j(p_j+q_j)$, $m=\sum_jK_j$,
$a=-q_r$, $b=-p_1$. Define positive $v_j$ by
$$
 \frac{\prod_j(z-q_j)}{z\prod_j(z-p_j)}
 =\frac{v_0^2}{z}+\sum_{j=1}^r\frac{v_j^2}{z-p_j},\enspace
 \sum_{j=0}^rv_j^2=1,\enspace v_0^2=\prod_jq_j/p_j\in(0,1),
$$
and set $v=(v_1,\ldots,v_r)^{\mathsf T}$,
$$
 \kappa_j=\left(\sum_k\frac{v_k^2}{(p_k-q_j)^2}
                  +\frac{v_0^2}{q_j^2}\right)^{-1/2},\enspace
 C_{jk}=\frac{\kappa_jv_k}{p_k-q_j},\enspace
 \eta_j=-\frac{\kappa_jv_0}{q_j}.
$$
Put $E=\operatorname{diag}(e^{i(K_jz-\Omega_jt-\Phi_j)})$.
\begin{theorem}\label{gas:arbitrary-band-quasiperiodic-gas}
For every $\Phi\in\mathbb T^r$, the denominator
$\tau_\Phi(t,z)=\det(I-EC)$ has no real zero, and
$u_\Phi=m-2\Im\partial_x\log\tau_\Phi$ is a global BO solution
in $\mathcal X_b$, with all derivatives bounded uniformly in $t,x,\Phi$.
Moreover
$$
 0<\frac{\min_jK_j}{4}(1-v_0^2)\le u_\Phi\le4b,\enspace
 \sigma(L_{u_\Phi})=\bigcup_j[p_j,q_j]\cup[0,\infty)
$$
is purely absolutely continuous of multiplicity one.
For Borel $I\subset(-\infty,0)$, put
$\nu_{R,\Phi,t}(I)=(2R)^{-1}\Tr[\mathbf1_{[-R,R]}
\mathbf1_I(L_{u_\Phi(t)})\mathbf1_{[-R,R]}]$. Then, uniformly in $t,\Phi$,
\begin{equation}\label{gas:arbitrary-band-window-dos}
 \left\|\frac{d\nu_{R,\Phi,t}}{d\lambda}
       -\frac1{2\pi}\sum_j\mathbf1_{[p_j,q_j]}\right\|_1\to0.
\end{equation}
\end{theorem}
\begin{proof}
Write $P=\operatorname{diag}(p_j)$, $Q=\operatorname{diag}(q_j)$,
$K=Q-P$. Evaluating and subtracting the residue formula
at $q_j,q_k$ shows that the rows $(C_{j1},\ldots,C_{jr},\eta_j)$
are an orthonormal basis of $(v,v_0)^\perp$. Therefore
$$
 \begin{pmatrix}C&\eta\\v^*&v_0\end{pmatrix}\text{ is unitary},\enspace
 C^*C=I-vv^*,\enspace CPC^*=Q.
$$
For $W=\operatorname{diag}(\sqrt{-p_j})$,
$\|W^{-1}CW\|=\rho:=\max_j\sqrt{q_j/p_j}<1$.
Thus $R=(I-EC)^{-1}$ and all its derivatives are uniformly bounded
on a common strip containing $\R$.
Put $\varphi=v^*R$ and $\Theta=v_0+\varphi E\eta$. Unitarity gives
\begin{equation}\label{gas:band-inner-kernel}
 1-\Theta(z)\overline{\Theta(w)}
 =\varphi(z)[I-E(z)E(w)^*]\varphi(w)^*.
\end{equation}
Hence $\Theta$ is inner with bounded derivative and
$\Theta(x+iy)\to v_0$ uniformly as $y\to\infty$.
On $\R$, direct multiplication gives
\begin{equation}\label{gas:band-resolvent-identities}
 H:=R+R^*-I=\varphi^*\varphi,\enspace RKR^*=P-RP-PR^*,
 \enspace u=\varphi K\varphi^*=-i\Theta'/\Theta.
\end{equation}
Since $\|EC\|\le1$, $\|R^*v\|\ge\|v\|/2$, proving the lower bound.
To verify the determinant formula of \cite[Theorem~1.1]{gasDK91},
put $f=\tau_x/\tau$, $h=\tau_t/\tau$.
The rank-one identity $H=\varphi^*\varphi$ and
\eqref{gas:band-resolvent-identities} give
$i(h-\bar h)+u^2-(f_x+\bar f_x)=\sum_j\Omega_j$.
The expansion $\log\tau=-\sum_{n\ge1}\Tr(EC)^n/n$ and its
derivatives of each fixed order converge uniformly, with positive frequencies.
Thus $|D|u=f_x+\bar f_x$, $u_t=i\partial_x(h-\bar h)$, and
spatial differentiation proves BO.

Define
$$
 \mathcal V:\bigoplus_jL^2(0,K_j)\to H_+,\enspace
 (\mathcal VF)(x)=\frac1{\sqrt{2\pi}}
 \sum_j\varphi_j(x)\int_0^{K_j}e^{isx}F_j(s)\,ds.
$$
The model-space kernel and \eqref{gas:band-inner-kernel} give
$\mathcal V\mathcal V^*=P_{K_\Theta}$.
Differentiating $R$ yields
$(D-u)\varphi=\varphi P+\kappa^{\mathsf T}E^*R^*$.
Column $j$ of $E^*R^*=\sum_{n\ge0}(E^*C^*)^nE^*$ has Fourier
support in $(-\infty,-K_j]$. Its product with a function supported
in Fourier space in $(0,K_j)$ is killed by $\Pi$, so
$L_u\mathcal V=\mathcal V\mathcal A$, $(\mathcal AF)_j(s)=(p_j+s)F_j(s)$.
These identities hold first for smooth compactly supported $F_j$,
then by density; bounded $\varphi_j,\varphi_j'$ ensure the operator
domain. The kernel of $\mathcal V$ reduces $\mathcal A$ and hence
splits over the disjoint bands. Each component is injective because
$\varphi_j$ is analytic and $\varphi_j(iy)\to v_j>0$.
Thus $\mathcal V$ is unitary onto $K_\Theta$. Lemma~\ref{gas:bounded-inner-model}
gives the remaining free spectrum; $L_u\ge-b$ and
Theorem~\ref{gas:inner-classification} give $u\in\mathcal X_b$ and $u\le4b$.

The transformed spectral projection has diagonal
$(2\pi)^{-1}\sum_j|I\cap[p_j,q_j]|\,|\varphi_j(x)|^2$.
Moreover
\begin{equation}\label{gas:band-channel-mean}
 (2R)^{-1}\int_{-R}^R|\varphi_j(t,x;\Phi)|^2dx\to1
 \enspace\text{uniformly in }t,\Phi.
\end{equation}
Indeed, componentwise isometry gives
$\int|\varphi_j(x;\Phi)|^2|f(x)|^2dx=\|f\|_2^2$ for
$\operatorname{supp}\widehat f\subset(0,K_j)$.
Average over any coset of the closure of $K\R$ in $\mathbb T^r$:
translation invariance and Fubini force the phase mean to be one.
Uniform trigonometric approximation converts it to the spatial mean,
since a character averages to $\sin((n\cdot K)R)/((n\cdot K)R)$
when $n\cdot K\ne0$. Time only changes the phase.
The diagonal formula now proves \eqref{gas:arbitrary-band-window-dos}.
\end{proof}

\begin{coro}
\label{gas:arbitrary-band-finite-arrays}\label{gas:inner-divisor-enclosure}
The preceding $\Theta$ is a pure Blaschke product. Its zeros
$X_\ell+iY_\ell$ satisfy $0<\delta\le Y_\ell\le Y<\infty$
and a uniform bound on the number of zeros whose real parts lie
in an interval of length one.
At time zero, the finite sums
$q_{R,\Phi}=\sum_{|X_\ell|\le R}2Y_\ell/((x-X_\ell)^2+Y_\ell^2)$
have simple negative eigenvalues $-a_{k,R,\Phi}$ with
$a_{k,R,\Phi}\in[a,b]$, and their BO solutions converge to $u_\Phi$
in $C^\infty_{\rm loc}(\R^2)$ uniformly in $\Phi$.
For every $g\in C([a,b])$, uniformly in $\Phi$,
$$
 \frac1{2R}\sum_k g(a_{k,R,\Phi})
 \longrightarrow\frac1{2\pi}\sum_j\int_{p_j}^{q_j}g(-\lambda)\,d\lambda.
$$
\end{coro}
\begin{proof}
The common strip and $\Theta(x+iy)\to v_0>0$ exclude singular
and exponential factors and bound the zero heights above and below.
The positive Poisson representation and $u\le4b$ bound their number
in intervals of length one; omitted poles contribute $O(R^{-1})$ locally,
with all derivatives. For a finite divisor $\Theta_R$, inclusion
$K_{\Theta_R}\subset K_\Theta$ and \eqref{gas:inner-conjugation}
put its compressed $D$ spectrum in $[a,b]$. Finite spectral theory
\cite{Sun21,BKV26} gives simplicity, and Theorem~\ref{gas:compact-flow}
gives flow convergence.

Choose real $\chi\in C_c^\infty$, $\int\chi^2=1$, put
$\chi_s(x)=\chi(x-s)$ and
$J_g(q;s)=\Tr[\chi_sg(-L_q)P_q\chi_s]$.
Then $\int J_g(q_{R,\Phi};s)ds=\sum_k g(a_{k,R,\Phi})$.
At distance $H$ from the endpoints, local trace continuity bounds
the truncation error by $\varepsilon(H)\to0$. Boundary strips
contribute $O(H)$; the exterior Poisson mass of retained poles is
$O(\log(2+R))$ by this counting bound. Consequently
$\sum_k g(a_{k,R,\Phi})=\int_{-R}^R J_g(u_\Phi(0);s)ds +O(R\varepsilon(H)+H+\log(2+R))$.
Use \eqref{gas:band-channel-mean}, divide by $2R$, and send first
$R\to\infty$, then $H\to\infty$.
\end{proof}

\section{Continuous norming data}
\label{gas:continuous-norming-section}
For finite data in \eqref{gas:finite-reconstruction} and real
$\gamma\in C([0,B])$, define $\mathcal T_\gamma$ by
$(a_j,c_j)\mapsto(a_j,c_j+\gamma(a_j))$.
\begin{theorem}
\label{gas:continuous-norming-action}
These maps extend uniquely to a jointly continuous action of
$C([0,B];\R)$, with the uniform norm, on $\mathcal X_B$, satisfying
$$
 \|\mathcal T_\gamma q-\mathcal T_\eta q\|_\infty
 \le8B^2\|\gamma-\eta\|_\infty,\enspace
 \mathcal T_{2t\,\mathrm{id}}=S(t),\enspace
 \mathcal T_dq=q(\cdot-d).
$$
If finite $q_n\to q$ locally smoothly and $\gamma_n\to\gamma$
uniformly, their modified solutions converge to $S(t)\mathcal T_\gamma q$
in $C^\infty_{\rm loc}(\R^2)$.
Write $K_q=K_{\Theta_q}$, $A_q=-L_q|_{K_q}$, $0\le A_q\le B$, and
$Y_q=X^*|_{K_q}$ on $K_q\cap\Dom X^*$, where $X^*$ is
$i\partial_\xi$ on $H^1(0,\infty)$ in Fourier variables.
For $z=x+iy$, $y>0$, set
$\mathcal R_{q,\gamma}(z)=(z-Y_q-\gamma(A_q))^{-1}P_q$
and $\mathcal R_q=\mathcal R_{q,0}$, extended by zero on $K_q^\perp$.
Their difference is trace class and
\begin{equation}\label{gas:norming-resolvent-reconstruction}
 \mathcal P_y(\mathcal T_\gamma q-q)(x)
 =-2\Im\Tr(\mathcal R_{q,\gamma}(x+iy)-\mathcal R_q(x+iy)).
\end{equation}
\end{theorem}
\begin{proof}
For finite data, $R=(x+B_N)^{-1}$ has
$\|R\|_{\mathrm{HS}}^2\le Bq(x)\le4B^2$ since
$\Im B_N\ge(2B)^{-1}I$. A center change $B_N\mapsto B_N-sD$
therefore gives $|\partial_sq|\le2\|D\|\|R\|_{\mathrm{HS}}^2$.

If $q_n\to q$, then $P_{q_n}\to P_q$ strongly,
$L_{q_n}\to L_q$ in the strong resolvent sense, and
$\Gamma_n:=\gamma(A_{q_n})\to\Gamma:=\gamma(A_q)$ strongly,
with uniformly bounded norms. For the last assertion use
$\Gamma_n=F(L_{q_n})P_{q_n}$ with $F\in C_c(\R)$ equal to
$\gamma(-\lambda)$ on $[-B,0]$.
Backward shifts preserve $K_q$; their restricted generator is
$-iY_q$. Hence $Y_q$ is maximal dissipative,
$\mathcal R_q(z)=(z-X^*)^{-1}P_q$, and the bounded self-adjoint
perturbation $\Gamma$ satisfies $\|\mathcal R_{q,\gamma}(z)\|\le y^{-1}$.
The formula
$$
 \mathcal R_q(z)g(w)=\frac{P_qg(w)-(P_qg)(z)}{z-w},\enspace
 \|M_{(z-x)^{-1}}P_q\|_{\mathrm{HS}}^2
 =\int\frac{q(x)}{2\pi|z-x|^2}\,dx
$$
gives Hilbert--Schmidt convergence of $M_{(z-x)^{-1}}P_{q_n}$:
strong projection convergence identifies the weak limit in $\Sch_2$,
and the integral gives convergence of the norms. The rank-one evaluation
terms converge because $P_{q_n}k_z\to P_qk_z$ in $L^2$, where $k_z$
is the Hardy reproducing kernel. Thus
$\mathcal R_{q_n}(z)\to\mathcal R_q(z)$ in $\Sch_2$.
Putting $R_n=\mathcal R_{q_n}(z)$, the identities
$\mathcal R_{q_n,\gamma}=(I-R_n\Gamma_n)^{-1}R_n$, $\|(I-R_n\Gamma_n)^{-1}\|\le1+\|\gamma\|_\infty/y, \enspace \mathcal R_{q_n,\gamma}-R_n=\mathcal R_{q_n,\gamma}\Gamma_nR_n$
give Hilbert--Schmidt convergence of perturbed resolvents and trace-norm
convergence of their differences. For finite data,
\cite[Proposition~5.4]{Sun21} gives $Y_q=-A^{1/2}B_NA^{-1/2}$;
the center change and the finite Poisson formula prove
\eqref{gas:norming-resolvent-reconstruction}.
Compactness and trace-norm convergence now determine a unique limit
of the modified potentials by its Poisson extension. This proves
continuity in $q$; the finite Lipschitz bound gives joint continuity.
The action and translation identities pass by density, and
Theorem~\ref{gas:compact-flow} gives convergence of the solutions.
\end{proof}

\begin{coro}
\label{gas:multiband-norming-phases}
For $q_\Phi=u_\Phi(0)$ from Theorem~\ref{gas:arbitrary-band-quasiperiodic-gas}
and real $\gamma\in C([0,b])$,
$\mathcal T_\gamma q_\Phi=q_{\Phi+\delta}$, where
$\delta_j=\int_{-q_j}^{-p_j}\gamma(a)\,da$.
Every phase vector is obtained from any fixed one in this way.
\end{coro}
\begin{proof}
Under $\mathcal V_\Phi$, $A_q$ is multiplication by $-p_j-s$ and
$Y_q$ becomes $Y_\Phi F=iF'$ on
$\{F_j\in H^1(0,K_j):F(K)=D_\Phi CF(0)\}$,
$D_\Phi=\operatorname{diag}(e^{-i\Phi_j})$.
Indeed, integration by parts gives
$x\mathcal V_\Phi F=\mathcal V_\Phi(iF')+iv^*F(0)/\sqrt{2\pi}$;
$X^*$ subtracts the last term. Dissipativity follows from
$C^*C=I-vv^*$, and the resolvent equation is invertible because
$\|\operatorname{diag}(e^{izK_j})D_\Phi C\|<1$ for $\Im z>0$.
Put $\Gamma_j(s)=\gamma(-p_j-s)$ and
$e_j(s,z)=\exp(-izs+i\int_0^s\Gamma_j)$.
Solving the first-order equation splits its resolvent into the
Volterra kernel $i\mathbf1_{r<s}e_j(s,z)/e_j(r,z)$ and a finite-rank
term. The Volterra resolvent difference has trace zero by its strict
triangular support, while the latter has trace
$-i\Tr[K(I-E_{\Phi+\delta}(z)C)^{-1}]$.
Subtract the case $\gamma=0$ and apply
\eqref{gas:norming-resolvent-reconstruction} to obtain the claimed phase
change. Continuous functions on disjoint bands can have arbitrary
prescribed integrals.
\end{proof}

\begin{prop}\label{prop:cm-bo-factor}
Let $(\lambda_j)_{j\in J}$ be a separated discrete real family and
suppose the matrix
$A_{jj}=p_j$, $A_{jk}=-i/(\lambda_j-\lambda_k)$ for $j\ne k$
extends to a bounded self-adjoint operator with $aI\le A\le bI$,
$0<a<b$. Put $Y=\diag(\lambda_j)$ on its natural domain and,
for real $\gamma\in C([a,b])$, set
$B_{\gamma,t}=Y-\gamma(A)-2tA+i/(2A)$,
$R_{\gamma,t}(z)=(z+B_{\gamma,t})^{-1}$.
For $\Im z>-(2b)^{-1}$, this resolvent is Hilbert--Schmidt and
$\|R_{\gamma,t}(z)\|\le(\Im z+(2b)^{-1})^{-1}$.
The function
\begin{equation}\label{gas:cm-bo-factor-field}
 u^\gamma(t,x)=\Tr[R_{\gamma,t}(x)^*A^{-1}R_{\gamma,t}(x)]
\end{equation}
is a global BO solution in $\mathcal X_b$, bounded by $4b$.
Every increasing finite exhaustion of $J$ gives finite multisolitons
converging in $C^\infty_{\rm loc}(\R^2)$; their zero-extended resolvents
converge in Hilbert--Schmidt norm locally uniformly on the stated strip.
This also holds for uniformly convergent finite norming functions.
For fixed $\gamma,t$,
\begin{equation}\label{gas:cm-resolution}
 \frac1{2\pi}\int_\R R_{\gamma,t}(x)^*A^{-1}R_{\gamma,t}(x)\,dx=I
\end{equation}
in the strong operator sense.
\end{prop}
\begin{proof}
For finite sections, $[A,Y]=i(ee^*-I)$, $e=(1,\ldots,1)^{\mathsf T}$,
gives $|e^*v|=1$ for every normalized eigenvector of $A$.
Thus its eigenvalues are simple. Normalize by $e^*v_j=1$; in this
basis $Y_{jk}=i/(a_j-a_k)$ for $j\ne k$, so $B_{\gamma,t}$ is
\eqref{gas:finite-reconstruction} with centers
$c_j=-Y_{jj}+\gamma(a_j)$. The identity
$R-R^*=-iR^*A^{-1}R$ proves \eqref{gas:cm-bo-factor-field}.

Separation implies $(Y+i)^{-1}\in\Sch_2$.
The real part of $B_{\gamma,t}$ is self-adjoint on $\Dom Y$;
coercivity for $z+B_{\gamma,t}$ and its adjoint proves invertibility
and the norm bound, and its resolvent identity proves membership
in $\Sch_2$. For finite projections $P_n$, extend
$A_n=P_nAP_n$ by $a_0(I-P_n)$, $a_0\in[a,b]$.
These extensions and their continuous functions converge strongly.
The resolvent identity, applied to a fixed Hilbert--Schmidt resolvent,
then gives Hilbert--Schmidt convergence; the complementary diagonal
block tends to zero because $(Y+i)^{-1}\in\Sch_2$.
Uniform inverse bounds give local uniformity and trace convergence.
Theorem~\ref{gas:compact-flow} identifies the resulting BO limit.
For $B=B_{\gamma,t}$,
$\|e^{isB}\|\le e^{-s/(2b)}$ and
$$
 \frac{d}{ds}(e^{-isB^*}e^{isB})=-e^{-isB^*}A^{-1}e^{isB},
 \enspace R(x)=-i\int_0^\infty e^{isx}e^{isB}\,ds.
$$
Integration and vector-valued Plancherel prove \eqref{gas:cm-resolution}.
\end{proof}

\begin{lemma}
\label{gas:cm-spectral-coordinate-chart}
Under Proposition~\ref{prop:cm-bo-factor}, suppose $J=\N_0$ and
$Y_{jj}=dj$, $d>0$. For fixed $t$ and real $\gamma\in C([a,b])$,
$H=Y-2tA-\gamma(A)$ on $\Dom Y$ has increasing simple eigenvalues
$\lambda_0<\lambda_1<\cdots\to\infty$, with
$|\lambda_j-dj|\le2b|t|+\|\gamma\|_\infty$.
Its normalized eigenvectors have unique phases satisfying
$e^*v_j:=\sum_{k\ge0}(v_j)_k=1$. If $Ue_j=v_j$ and $A^\sharp=U^*AU$, then
$\lambda_{j+1}-\lambda_j\ge\frac2{b-a}$, $A^\sharp_{jj}\in[a,b]$, $A^\sharp_{jk}=-\frac{i}{\lambda_j-\lambda_k}\enspace(j\ne k)$.
The solution is reconstructed by
$u^\gamma(t,x)=\Tr[\mathcal R(x)^*(A^\sharp)^{-1}\mathcal R(x)]$,
where $\mathcal R(x)=(x+\diag(\lambda_j)+i/(2A^\sharp))^{-1}$.
\end{lemma}
\begin{proof}
The functional $e^*$ is continuous in the graph norm of $Y$.
Approximation by finite vectors extends the commutator to
$(Au)^*Hv-(Hu)^*Av=i(\overline{e^*u}\,e^*v-u^*v), \enspace u,v\in\Dom Y$,
since the bounded terms commute with $A$.
Compact resolvent and min--max give the eigenvalue estimate.
The identity yields $|e^*v|=1$ for every normalized eigenvector,
proving simplicity and the phase choice; on distinct eigenvectors
it gives the off-diagonal entries. The bound
$\|A^\sharp-(a+b)I/2\|\le(b-a)/2$ gives the gap.
Unitary conjugation of $B_{\gamma,t}$ proves reconstruction.
\end{proof}

\section{A one-sided gas with continuous norming data}
\label{gas:one-sided-continuum}\label{sec:cm-physical-negative-transform}\label{gas:sawtooth-explicit-spectral}

Fix $0<a<b$, $K=b-a$, $d=2\pi/K$ and $p=(a+b)/2$. On
$\ell^2(\mathbb N_0)$ set
$Ye_j=dj e_j$, $A_{jj}=p$, $A_{jk}=-\frac{i}{d(j-k)}\ (j\ne k), \enspace\mathcal D(Y)=\{v:(jv_j)\in\ell^2\}$.
The Laurent symbol and bounds are
\begin{equation}\label{gas:half-lattice-symbol}
 \alpha(\theta)=p+\frac{\pi-\theta}{d},\enspace0<\theta<2\pi,
 \enspace aI\le A\le bI.
\end{equation}
For real $\gamma\in C([a,b])$, put $\Gamma=\gamma(A)$ and, on
$\mathcal D(Y)$,
$$
 \begin{gathered}
 B_{\gamma,t}=Y-\Gamma-2tA+\tfrac i2A^{-1},\enspace
 R_{\gamma,t}(x)=(x+B_{\gamma,t})^{-1},\\
 u^\gamma(t,x)=\Tr(R_{\gamma,t}(x)^*A^{-1}R_{\gamma,t}(x)).
 \end{gathered}
$$
Any continuous extension of $\gamma$ to $[0,b]$ gives the same
action in Theorem~\ref{gas:continuous-norming-action}, since all finite
spectra lie in $[a,b]$.
For $\gamma_r(s)=\log r(s)/(2s)$, the notation in
Theorem~\ref{thm:main-gas} is $u_r=u^{\gamma_r}$ and
$u_1=u^0$.

\begin{theorem}
\label{gas:half-lattice-gas}
The function $u^\gamma$ is a global BO solution in $\mathcal X_b$,
$0\le u^\gamma\le4b$. For the projection $P_N$ onto indices
$0,\ldots,N-1$, the eigenvalues $a_{j,N}$ of $A_N=P_NAP_N$ are
simple and lie in $[a,b]$. The multisolitons with data
$c_{j,N}^\gamma=-d(N-1)/2+\gamma(a_{j,N})$ converge to $u^\gamma$
in $C^\infty_{\rm loc}(\mathbb R^2)$, and
$$
 N^{-1}\sum_{j=1}^N\delta_{a_{j,N}}
 \rightharpoonup K^{-1}\mathbf1_{[a,b]}(s)\,ds,
 \enspace u^\gamma(t)=S(t)\mathcal T_\gamma u^0(0).
$$
For every positive $r\in C([a,b])$, $\gamma_r(s)=\log r(s)/(2s)$ gives
$n_{j,N}^{(r)}=e^{2a_{j,N}c_{j,N}^{\gamma_r}} =e^{-d(N-1)a_{j,N}}r(a_{j,N})$.
The map $r\mapsto u^{\gamma_r}(0)$ is injective.
\end{theorem}
\begin{proof}
The Fourier series of $\alpha$ gives~\eqref{gas:half-lattice-symbol}.
Proposition~\ref{prop:cm-bo-factor} gives reconstruction, convergence,
the bound $\|R_{\gamma,t}(x)\|\le2b$ and simplicity of $A_N$.
Index reversal followed by conjugation commutes with $A_N$; simplicity
therefore gives symmetric eigenvector moduli and $Y$-expectation
$d(N-1)/2$. This proves the center and norming formulas; the action
follows from Theorem~\ref{gas:continuous-norming-action}.

For the eigenvalue distribution, replace $\alpha$ by real Fej\'er
sums $\alpha_m\in[a,b]$. The normalized Hilbert--Schmidt error of
their Toeplitz sections is bounded by $\|\alpha-\alpha_m\|_2$.
Telescoping powers transfers this bound to polynomial traces;
finite-band paths contribute only $O(1)$ endpoint errors. Thus
$N^{-1}\Tr f(A_N)\to(2\pi)^{-1}\int f(\alpha(\theta))\,d\theta
=K^{-1}\int_a^bf(s)\,ds$, first for polynomials and then for
continuous $f$. Injectivity follows from
Theorem~\ref{gas:half-lattice-weight-response}.
\end{proof}

Put $\ell=2\pi/K$, $dm=d\theta/(2\pi)$,
$\omega(e^{i\theta})=a+K\theta/(2\pi)$, and $A_+=T_\omega$ on
$H^2(\mathbb T)$. Set $H(w)=(1+w)/(1-w)$ and, for $a<s<b$,
$$
 d_s=(s-a)/K,\enspace\vartheta_s=2\pi d_s,\enspace
 g_s(w)=\exp\left[-\tfrac12\int\log|\omega(\zeta)-s|
 H(w\bar\zeta)\,dm(\zeta)\right].
$$
Analytic square roots below equal one at zero.
\begin{prop}
\label{gas:sawtooth-spectral-transform}
For $k_u(w)=(1-\bar uw)^{-1}$, the formula
$$
 Uk_u(s)=\overline{\phi(u;s)},\enspace
 \phi(w;s)=\sqrt{\frac{\sin(\pi d_s)}\pi}
 \frac{g_s(w)}{\sqrt{(1-w)(1-we^{-i\vartheta_s})}}
$$
defines a unitary $U:H^2(\mathbb T)\to L^2(a,b)$ with
$UA_+U^*=M_s$. Hence $A_+$ has simple purely absolutely continuous
spectrum $[a,b]$.
\end{prop}
\begin{proof}
Set $F_z=\exp[\frac12\int\log(\omega-z)H\,dm]$, continuing the
real logarithm from $z<a$. Outer factorization gives
$\langle(A_+-z)^{-1}k_u,k_v\rangle
=F_z(v)^{-1}\overline{F_{\bar z}(u)^{-1}}/(1-\bar uv)$.
Continuation to the cut gives
$$
 F_{s\pm i0}(w)^{-1}
 =g_s(w)e^{\pm i\pi d_s/2}
 \left(\frac{1-we^{-i\vartheta_s}}{1-w}\right)^{\pm1/2}.
$$
Subtracting the boundary values and using Stone's formula gives
the spectral density $\phi(v;s)\overline{\phi(u;s)}$.
For fixed disk points, boundedness of $H$ and the uniform bound on
$\int|\log|\omega-z||\,dm$ bound the scalar resolvents near $[a,b]$.
Their Poisson kernels give $\mu(I)\le C|I|$, including endpoints;
the form bounds exclude exterior spectrum. Integrating the density
therefore makes $U$ an isometry. Its closed range reduces spectral
multiplication (the compressed spectral projections have the same
norms as the full ones) and contains $U1=\phi(0;\cdot)>0$ almost
everywhere. Bounded multiples of this function are dense, proving
surjectivity.
\end{proof}

Let $J$ be coefficientwise conjugation on $H^2$ and $C$ conjugation on $L^2(a,b)$.
Then $A=JA_+J$, $U_A=CUJ$ is unitary and $U_AAU_A^*=M_s$.
\begin{theorem}
\label{thm:cm-physical-negative-transform}
Fix $t\in\mathbb R$ and real $\gamma\in C([a,b])$. Write
$\mathcal B=B_{\gamma,t}$, $R(x)=(x+\mathcal B)^{-1}$,
$q=u^\gamma(t)$ and $e=(1,1,\ldots)^{\mathsf T}$. Define
$(Wf)(x)=(2\pi)^{-1/2}e^*R(x)A^{-1/2}f$.
Then $W$ is an isometry onto the negative spectral subspace of $L_q$, with
$L_qW=-WA$ and $WW^*=P_q$. The full spectrum is purely absolutely
continuous, equal to $[-b,-a]\cup[0,\infty)$, with multiplicity one
on $[-b,-a]$. For bounded Borel $h$ on $[a,b]$,
\begin{equation}\label{eq:cm-negative-spectral-kernel}
 [h(-L_q)P_q](x,y)=(2\pi)^{-1}e^*R(x)A^{-1}h(A)R(y)^*e.
\end{equation}
\end{theorem}
\begin{proof}
For $h_0=(Y+I)^{-1}e\in\ell^2$, $e^*$ is continuous in the graph norm of
$\mathcal D(Y)$, and, with $V=\mathcal B-Y$,
\begin{equation}\label{eq:cm-observation-row-domain}
 R(x)^*e=h_0+R(x)^*[(1-x)I-V^*]h_0\in\ell^2.
\end{equation}
Apply the resolvent equation to finite truncations of $h_0$ and
pass to the limit on the right in $\ell^2$ to justify this identity.
The commutator $[A,Y]=i(ee^*-I)$ extends from finite support as a
form on $\mathcal D(Y)$, giving
$A\mathcal B-\mathcal B^*A=i ee^*$.
The semigroup $V_s=e^{i\mathcal B s}$ preserves this domain and
decays at rate $e^{-s/(2b)}$. Differentiating
$\langle V_sg,AV_sg\rangle$ and integrating yields
$\int_0^\infty|e^*V_sg|^2\,ds=\langle g,Ag\rangle$.
The Laplace formula in the graph norm
$R(x)g=-i\int_0^\infty e^{ixs}V_sg\,ds$ and Plancherel make $W$
an isometry on $A^{1/2}\mathcal D(Y)$, hence on $\ell^2$, into $H_+$.

For finite sections, diagonalize $A_N$ with $e_N^*u_j=1$ and set
$G_N=A_N^{1/2}\mathcal B_NA_N^{-1/2}$, $v_N=A_N^{-1/2}e_N$.
Equations~\eqref{eq:finite-center-Wu}--\eqref{eq:finite-X-offdiagonal}
give the negative eigenfunction row
$\widehat\varphi_N(\xi)=\sqrt{2\pi}v_N^*e^{iG_N\xi}$;
Fourier inversion gives $\varphi_N=iW_N$.
Consequently $W_N^*W_N=I_N$, $L_{q_N}W_N=-W_NA_N$ and
$q_N(x)=\|A_N^{-1/2}R_N(x)^*e_N\|^2$, the last identity also
following from the resolvent identity for $G_N-G_N^*=iv_Nv_N^*$.

Proposition~\ref{prop:cm-bo-factor} and
\eqref{eq:cm-observation-row-domain} give convergence of the
resolvent rows locally uniformly in $x$. Thus $W_NP_Nf\to Wf$
locally in $L^2$, globally weakly, and then strongly by equality of
norms. Also $W_NA_NP_Nf\to WAf$.
The uniformly bounded potentials $q_N\to q$ locally uniformly,
so $L_{q_N}\to L_q$ in strong resolvent sense. Passing through
$(L_{q_N}-i)^{-1}$ proves $L_qW=-WA$ and $WW^*\le P_q$.
The row identity gives $q(x)=\|A^{-1/2}R(x)^*e\|^2$; hence, for
$\chi\in C_c^\infty(\mathbb R;\mathbb R)$, Lemma~\ref{gas:bounded-inner-model} gives
$\Tr(\chi WW^*\chi)=(2\pi)^{-1}\int\chi^2q =\Tr(\chi P_q\chi)$.
Positivity implies $WW^*=P_q$. Functional calculus proves
\eqref{eq:cm-negative-spectral-kernel}; equivalence to $-M_s$ on
this range and to $D$ on its complement proves the spectral claims.
\end{proof}

\section{Spatial and long-time asymptotics}
\label{gas:one-sided-exterior-asymptotics}
We use the notation of Section~\ref{gas:one-sided-continuum}, with
$\Gamma=\gamma(A)$, $C=(2A)^{-1}$ and
$\alpha(\theta)=p+(\pi-\theta)/d$.

\begin{theorem}
\label{gas:half-lattice-far-fields}
Uniformly for $|t|\le T_0<\infty$,
$$u^\gamma(t,x)=\frac{\log(b/a)}{2\pi x}
 +O_{a,b,T_0,\|\gamma\|_\infty}(x^{-2}\log(2+x)),\enspace x\to+\infty.$$
Put $\varrho=\sqrt{a/b}$ and $\delta_\gamma=\int_a^b\gamma(s)\,ds$.
For every $m,n\ge0$, $\partial_x^m\partial_t^n(u^\gamma-Q_\gamma)
\to0$ as $x\to-\infty$, uniformly for $|t|\le T_0$, where
\begin{equation}\label{gas:half-lattice-background}
 Q_\gamma(t,x)=K\frac{1-\varrho^2}
 {1-2\varrho\cos(Kx-K(a+b)t-\delta_\gamma)+\varrho^2}.
\end{equation}
\end{theorem}
\begin{proof}
For $D_x=(x+Y)^{-1}$ and $V=-2tA-\Gamma+iC$, the resolvent identity gives
$\|R-D_x\|_{\mathfrak S_2}\le2\|V\|x^{-1}\|D_x\|_{\mathfrak S_2}$
when $x\ge2\|V\|$, with
$\|D_x\|_{\mathfrak S_2}^2\le x^{-2}+(dx)^{-1}$. Hence
$u^\gamma=\sum_{j\ge0}(A^{-1})_{jj}(x+dj)^{-2}+O(\|V\|x^{-2})$.
The Laurent coefficients of $A_\infty^{-1}$ satisfy
$c_0=d\log(b/a)/(2\pi)$ and $c_j=O(|j|^{-1})$.
The inverse block formula yields
$0\le c_0-(A^{-1})_{jj}\le b\sum_{k<0}|c_{k-j}|^2\le C/(j+1)$.
Since $\sum_{j\ge0}(x+dj)^{-2}=(dx)^{-1}+O(x^{-2})$ and
$\sum_{j\ge0}(j+1)^{-1}(x+dj)^{-2}=O(x^{-2}\log(2+x))$,
the first assertion follows.

For $x=-dN+z$, $0\le z<d$, recenter the half lattice at $N$.
The compressed $A_\infty$, extended by $a_0I$ with $a_0\in[a,b]$,
converges strongly to $A_\infty$, as do its inverse and continuous
functions. The resolvent identity and $\|R\|\le2b$ upgrade this to
Hilbert--Schmidt convergence: the bilateral reference resolvent is
Hilbert--Schmidt and the complementary block has squared norm
$O(N^{-1})$. The convergence is uniform on compact $(t,z)$-sets
and survives all derivatives, which are products of resolvents and
bounded functions of $A$.

The bilateral operator is $B_\infty=-id\partial_\theta+V(\theta)$,
where $V=-2t\alpha-\gamma(\alpha)+i/(2\alpha)$ has average
$V_{\rm av}=-2tp-\bar\gamma+iy_*$,
$\bar\gamma=K^{-1}\int_a^b\gamma$ and $y_*=d\log(b/a)/(4\pi)$.
The periodic Lipschitz multiplier
$G=\exp[-(i/d)\int_0^\theta(V-V_{\rm av})]$ gives
$B_\infty=G(Y_\infty+V_{\rm av})G^{-1}$.
For $R_0=(z+Y_\infty+V_{\rm av})^{-1}$,
$GR_0G^{-1}-R_0=-GR_0[Y_\infty,G^{-1}]R_0$ is trace class with
zero trace: in Fourier coordinates each inner sum
$\sum_j[(R_0)_{j-\ell,j-\ell}-(R_0)_{jj}]$ telescopes to zero;
its absolute sum is $O(|\ell|)$ and $G,G^{-1}\in H^1$ justify interchange.
Consequently $R-R^*=-iR^*A_\infty^{-1}R$ gives
\begin{equation}\label{eq:half-cm-bilateral-wave}
 \operatorname{Tr}(R^*A_\infty^{-1}R)
 =\sum_{j\in\mathbb Z}\frac{2y_*}
 {(z+dj-2pt-\bar\gamma)^2+y_*^2}.
\end{equation}
Poisson summation proves \eqref{gas:half-lattice-background}.
\end{proof}

\begin{theorem}
\label{gas:half-lattice-exterior}
Uniformly for $\xi$ in each compact $J\Subset(2b,\infty)$,
$$u^\gamma(T,T\xi)=\frac{F(\xi)}T+O_J(T^{-3/2}),\enspace
 F(\xi)=\frac1{2\pi\xi}\log\frac{b(\xi-2a)}{a(\xi-2b)},\enspace T\to\infty.$$
The constant may depend on $a,b,J,\|\gamma\|_\infty$.
\end{theorem}
\begin{proof}
Set $H_T=Y/T+\xi-2A$, $S=H_T^{-1}$ and
$S_\gamma=(H_T+T^{-1}(-\Gamma+iC))^{-1}$.
The Neumann factorization using $Y/T+\xi$ and the resolvent identity give
\begin{equation}\label{gas:half-lattice-column-bound}
 \|S_\gamma e_j\|+\|Se_j\|\le\frac{C_J}{1+j/T},\enspace
 \|(S_\gamma-S)e_j\|\le\frac{C_J}{T(1+j/T)}.
\end{equation}
Also $\|S_\gamma-S\|\le C_J/T$.
Thus $\|S_\gamma\|_{\mathfrak S_2}+\|S\|_{\mathfrak S_2}=O_J(T^{1/2})$
and $\|S_\gamma-S\|_{\mathfrak S_2}=O_J(T^{-1/2})$, whence
$Tu^\gamma=T^{-1}\sum_j\langle Se_j,A^{-1}Se_j\rangle+O_J(T^{-1})$.

Recenter at $j$, put $r=j/T$, $P=\mathbf1_{k\ge-j}$,
$A_j=PA_\infty P$ and $z=(\xi+dr-2A_\infty)^{-1}e_0$.
Bounded variation of the symbol gives
$|z_k|\le C_J(1+r)^{-2}|k|^{-1}$ for $k\ne0$.
For $j\ge2$ truncate at $M=\min(\lfloor j/2\rfloor,\lfloor T\rfloor)$,
writing $w=\mathbf1_{|k|\le M}z$ and $Y'e_k=dke_k$.
The identity
$(Y'/T+\xi+dr-2A_j)w-e_0
=P(\xi+dr-2A_\infty)(w-z)+T^{-1}Y'w$
and coercivity give
$$\|Se_j-Pz\|\le C_J\left[\frac{M^{-1/2}}{1+r}
 +\frac{\sqrt M}{T(1+r)^2}\right],$$
where the left side is understood in recentered coordinates.
Likewise, for $v=A_\infty^{-1}z$,
$\|(I-P)v\|\le C_J(1+r)^{-1}(j+1)^{-1/2}$ and
$A_j^{-1}Pz=Pv+A_j^{-1}PA_\infty(I-P)v$. Hence
$$\big|\langle Se_j,A^{-1}Se_j\rangle-\langle z,A_\infty^{-1}z\rangle\big|
 \le\frac{C_J(M^{-1/2}+(j+1)^{-1/2})}{(1+r)^2}
       +\frac{C_J\sqrt M}{T(1+r)^3}.$$
Its normalized sum is $O_J(T^{-1/2})$, by splitting at $j=T$;
the first two indices contribute $O(T^{-1})$.
The bilateral diagonal is
$f_\xi(r)=(2\pi)^{-1}\int_0^{2\pi}
[\alpha(\theta)(\xi+dr-2\alpha(\theta))^2]^{-1}\,d\theta$.
Since $|f_\xi'|\le C_J(1+r)^{-3}$, its Riemann-sum error is $O_J(T^{-1})$.
Since $dK=2\pi$, this gives
$\int_0^\infty f_\xi(r)\,dr=(2\pi)^{-1}\int_a^b[s(\xi-2s)]^{-1}ds=F(\xi)$.
\end{proof}

\begin{theorem}
\label{gas:half-lattice-weight-response}
For $\gamma,\eta\in C([a,b];\mathbb R)$, locally uniformly for $\xi>2b$,
\begin{equation}\label{gas:half-lattice-response}
 T^2(u^\gamma-u^\eta)(T,T\xi)\longrightarrow
 \frac1{2\pi}\int_a^b\frac{\gamma(s)-\eta(s)}{s(\xi-2s)^2}\,ds,
 \enspace T\to\infty.
\end{equation}
The map $\gamma\mapsto u^\gamma(0,\cdot)$ is injective.
\end{theorem}
\begin{proof}
For $\eta=0$, use $S_\gamma-S_0=T^{-1}S_\gamma\Gamma S_0$ in the
quadratic trace. By \eqref{gas:half-lattice-column-bound}, replacing
all resolvents by $S$ costs $O_J(T^{-1})$ in the normalized trace, giving
$$T^2(u^\gamma-u^0)
 =\frac2T\operatorname{Re}\sum_{j\ge0}
 \langle Se_j,A^{-1}S\Gamma Se_j\rangle+O_J(T^{-1}).$$
For $j/T\to r>0$, the recentered $A$, $A^{-1}$ and $\gamma(A)$ converge
strongly to their bilateral counterparts. The uniform Neumann series
gives $S\to(\xi+dr-2A_\infty)^{-1}$. The summand therefore tends to
$(2\pi)^{-1}\int_0^{2\pi}\gamma(\alpha)/[\alpha(\xi+dr-2\alpha)^3]\,d\theta$.
The bound $C_J(1+j/T)^{-2}$ controls both ends of the Riemann sum;
convergence is uniform on compact $(r,\xi)$-sets by the same sequential
argument. Integrating in $r$ proves \eqref{gas:half-lattice-response}.
Replacing $\Gamma$ by $\gamma(A)-\eta(A)$ proves the general case.
If the initial fields agree, uniqueness of the BO flow makes this
transform zero. Its expansion at infinity gives
$\int_a^b(\gamma-\eta)(s)s^{n-1}\,ds=0$ for all $n\ge0$;
polynomial density and continuity imply $\gamma=\eta$.
\end{proof}

\section{Long-time asymptotics in the periodic region}
\label{gas:one-sided-gas-ray-section}
\begin{prop}
\label{prop:half-cm-gap-deficit}
The increasing eigenvalues of $Y-2tA$ satisfy
$\lambda_{j+1}(t)-\lambda_j(t)\ge d$ for every $j\ge0$ and $t\in\mathbb R$.
\end{prop}
\begin{proof}
The shift identity $S^*(Y-2tA)S=Y-2tA+dI$ gives the assertion by
codimension-one min--max interlacing.
\end{proof}

\subsection{The transmission coefficient for the step operator}
Let $D=-i\partial_y$ and let $\Pi_\pm$ be the projections onto the
positive and negative Fourier half-lines, with the Fourier convention
of Section~\ref{sec:lax-data}.
\begin{lemma}
\label{lem:hardy-step-scattering}
For $\kappa>h>0$, set $s_\pm=\kappa\pm h$,
$\nu=(2\pi)^{-1}\log(s_+/s_-)$ and
$\Phi(\kappa)=\pi^{-1}\operatorname{arctanh}^2(h/\kappa)$.
There is a bounded positive-frequency generalized eigenfunction
$w_\kappa\in H^1_{\rm loc}$ satisfying
\begin{equation}\label{gas:hardy-step-equation}
 (D-h\Pi_+\operatorname{sgn}(y))w_\kappa=\kappa w_\kappa.
\end{equation}
For $\pm y>0$, $w_\kappa(y)=C_\pm e^{is_\pm y}+r_\kappa(y)$, where
$|C_\pm|=1$,
$C_+/C_-=e^{i\Phi(\kappa)}$,
and, with all differentiated estimates,
\begin{equation}\label{gas:hardy-step-mellin-tail}
 r_\kappa(y)=D_\pm|y|^{-1-i\nu}+O(|y|^{-2}\log|y|),
 \enspace y\to\pm\infty.
\end{equation}
For $F=\widehat w_\kappa$ and
$a_\kappa=\Pi_-(\operatorname{sgn}(y)w_\kappa)$,
$$
 F(s)=A_\kappa s^{i\nu}(1+O(s|\log s|)),\enspace
 \widehat a_\kappa(-s)=B_\kappa s^{i\nu}(1+O(s|\log s|)),\enspace s\to0.
$$
The endpoint estimates may be differentiated, and all estimates are
locally uniform for $\kappa>h$.
\end{lemma}
\begin{proof}
On the plane cut along $(-\infty,0]$, define
\begin{equation}\label{gas:hardy-step-factor}
 Z(z)=\exp\left\{\frac1{2\pi i}\int_{-\infty}^0
 \frac{\log((u-s_+)/(u-s_-))}{u-z}\,du\right\}.
\end{equation}
Its jump is $(u-s_+)/(u-s_-)$, so the functions
$[Z(z)(z-s_-)]^{-1}$ above the real axis and
$[Z(z)(z-s_+)]^{-1}$ below it define $G$ analytic across the negative
axis. Set $F=G_+-G_-$, interpreting poles as boundary distributions,
and $w_\kappa=\mathcal F^{-1}F$. The Cauchy formula is
$G(z)=(2\pi i)^{-1}\int_0^\infty F(s)/(s-z)\,ds$.
Indeed $G(z)=z^{-1}+O(z^{-2}\log|z|)$ at infinity, and its
imaginary-power singularity at zero is bounded in fixed angular
sectors, so the corresponding contour contributions vanish;
the poles use $(u\pm i0)^{-1}=\operatorname{pv}(1/u)\mp i\pi\delta_0$.
Plemelj's identities give
$(s-\kappa+h)G_+-(s-\kappa-h)G_-=0$ for $s>0$,
which is the Fourier transform of \eqref{gas:hardy-step-equation}.
The pole terms give
\begin{equation}\label{gas:hardy-step-amplitudes}
 C_\pm=-i/Z(s_\pm),\enspace |C_\pm|=1,
\end{equation}
since $\log Z$ is purely imaginary on the positive axis.

Splitting the logarithm in \eqref{gas:hardy-step-factor} at its
endpoint value $2\pi\nu$ gives
$Z(s)=Z_0s^{-i\nu}(1+O(s|\log s|))$, $|Z_0|=1$.
Thus $F(s)=-2h s^{i\nu}(1+O(s|\log s|))/[(\kappa^2-h^2)Z_0]$;
the identity $\widehat a_\kappa(s)=-2G(s)$ for $s<0$ gives the
second endpoint expansion. Both follow with differentiated bounds.
Subtract localized pole distributions. Their inverse transforms differ
from the half-line plane waves by rapidly decaying terms. The remaining
endpoint powers use the Abel-regularized integral
$\mathcal F^{-1}(\mathbf1_{s>0}s^{k+i\nu})(y)
=(2\pi)^{-1}\Gamma(k+1+i\nu)e^{\pm i\pi(k+1+i\nu)/2}|y|^{-k-1-i\nu}$
for $\pm y>0$. This gives \eqref{gas:hardy-step-mellin-tail}.
At infinity, $\partial_s^jF(s)=O(s^{-2-j})$;
after a high-frequency cutoff, arbitrarily high integrable derivatives
of $(is)^kF$ give rapid decay away from $y=0$ by integration by parts.
This also gives boundedness of $w_\kappa$ and at most a logarithmic
singularity of $w_\kappa'$ at zero, hence $H^1_{\rm loc}$.

With $L(v)=\log((v+h)/(v-h))$ and
$L'(v)=-2h/(v^2-h^2)$,
$$\log\frac{Z(s_-)}{Z(s_+)}
=-\frac1{2\pi i}\int_\kappa^\infty
 L(v)\left(\frac1{v-h}-\frac1{v+h}\right)dv
=-\frac{L(\kappa)^2}{4\pi i}=i\Phi(\kappa).$$
Together with \eqref{gas:hardy-step-amplitudes}, this proves the phase.
\end{proof}

\subsection{Eigenvalue asymptotics and reconstruction of the solution}
\label{subsec:gas-side-curvature}
Put $h=2\pi/d^2$, $c=h/\pi$ and
$V(\theta)=-h\operatorname{sgn}\theta+c\theta$ on $(-\pi,\pi)$.
With $D=-i\partial_\theta$, $P$ the projection onto nonnegative
modes and $Q=I-P$, set
$L_t=P(D+tV)P=d^{-1}(Y-2tA+2ptI)$; its gaps are at least one.
\begin{prop}
\label{prop:gas-side-eigenvalue-marks}
Let $(\lambda_n(t),v_n(t))$ be the increasing normalized eigenpairs
of $Y-2tA$ and put
$\Delta(\kappa)=-d\operatorname{arctanh}^2(h/\kappa)/(2\pi^2)$.
For every compact interval $J\Subset(h,\infty)$ with nonempty interior
and every $f\in C([a,b])$, as $t\to+\infty$,
\begin{equation}\label{eq:gas-side-eigenvalue-asymptotics}
 \sup_{n/t\in J}|\lambda_n(t)-dn+2pt-\Delta(n/t)|\to0,
\end{equation}
\begin{equation}\label{eq:gas-side-spectral-marks}
 \sup_{n/t\in J}\left|\langle f(A)v_n(t),v_n(t)\rangle
 -\frac1{b-a}\int_a^bf(s)\,ds\right|\to0.
\end{equation}
\end{prop}
\begin{proof}
Use the preceding step eigenfunction and put
$p_\kappa(y)=C_{\operatorname{sgn}y}e^{i(\kappa y+h|y|)}$,
$r_\kappa=w_\kappa-p_\kappa$.
Choose $\chi\in C_c^\infty(-\pi,\pi)$ equal to one near zero and impose
\begin{equation}\label{eq:gas-ray-monodromy}
 \kappa t+\Phi(\kappa)/(2\pi)=n\in\mathbb Z.
\end{equation}
For $q_t=e^{-itc\theta^2/2}$, define
$u_t=q_tp_\kappa(t\theta)+\chi r_\kappa(t\theta)$.
Its jumps at zero cancel, and \eqref{eq:gas-ray-monodromy} matches
all endpoint derivatives since $c\pi=h$.
Thus $u_t\in H^1(\mathbb T)$,
$\|u_t-q_tp_\kappa(t\cdot)\|_2=o(1)$, and its norm stays positive.

The endpoint expansions imply
\begin{equation}\label{eq:gas-ray-negative-projection}
 t^{1+i\nu}Qu_t\longrightarrow b_\kappa
 \enspace\hbox{strongly in }L^2(\mathbb T).
\end{equation}
Indeed, decompose
$u_t=\chi w_\kappa(t\cdot)+\chi(q_t-1)p_\kappa(t\cdot)
 +(1-\chi)q_tp_\kappa(t\cdot)$.
For the first term, its $j<0$ Fourier coefficient is
$(4\pi^2)^{-1}\int_0^\infty F(s)\widehat\chi(j-ts)ds$.
After $v=ts$, the scaled coefficient tends to
$(4\pi^2)^{-1}A_\kappa\int_0^\infty v^{i\nu}\widehat\chi(j-v)dv$.
The bound $C_N(1+|j|)^{-N}$ gives convergence in weighted $\ell^2$;
parts away from zero, including localized pole distributions, decay
rapidly. For the second term, three nonstationary integrations by
parts give $Ct/(t+|j|)^3$, hence negative projection $O(t^{-3/2})$:
its first endpoint terms cancel because $q_t-1$ vanishes to second
order at zero. The last term is $O(t^{-N})$ by positive phase
derivative and cancellation of periodic endpoint terms.

With $a_\kappa=\Pi_-^{\mathbb R}(\operatorname{sgn}(\cdot)w_\kappa)$,
the exact residual is
$$R_t:=(L_t-\kappa t)Pu_t
=-thP(\chi a_\kappa(t\cdot))+tcP(\theta\chi r_\kappa(t\cdot))
 -iP(\chi'r_\kappa(t\cdot))-tPVQu_t.$$
The family $t^{i\nu}R_t$ converges strongly. The first term uses
the negative-frequency endpoint expansion and the same cutoff
argument; the second uses dominated convergence from
$|t\theta r_\kappa(t\theta)|\le C$ and its Mellin limit.
The third is $O(t^{-1})$, and the last uses
\eqref{eq:gas-ray-negative-projection}. Hence $\{R_t\}$ is precompact.

Choose $0<\eta<\kappa-h$, $M=\lfloor\eta t\rfloor$ and let
$P_M$ project onto $0\le j\le M$. Set
$C_t=P_M(D/t+T_V-\kappa)P_M$, $g_t=C_t^{-1}P_MR_t$ and
$\psi_t=Pu_t-t^{-1}g_t$.
Since $C_t\le-(\kappa-h-\eta)I$, the extended inverses converge
strongly to $(T_V-\kappa)^{-1}$. Precompactness of $R_t$ implies
that of $g_t$, so
$(L_t-\kappa t)\psi_t=(I-P_M)R_t-(I-P_M)T_Vg_t\to0, \enspace \|g_t/t\|_2=O(t^{-1})$.
These arguments hold uniformly on compact $\kappa$-sets by passing
to convergent parameter subsequences and retaining the scalar phase
$t^{i\nu(\kappa)}$.

The gap bound gives a unique eigenvalue within $o(1)$ of $\kappa t$.
To fix its label, replace $(h,c)$ by $(\sigma h,\sigma c)$,
$0\le\sigma\le1$. The poles $\kappa\pm\sigma h$ stay in a compact subset of $(0,\infty)$,
so a common cutoff contains both and avoids zero.
The factors $Z_\sigma$ and their derivatives are uniform there;
translations of $(s-s_\pm\mp i0)^{-1}$ have uniformly bounded
distribution seminorms, including when the poles coalesce. The endpoint and tail coefficients of $F,r$
vanish at zero coupling; the $a$ residual is multiplied by $\sigma h$.
Thus the estimates extend uniformly to $\sigma=0$, where the trial
function is the $n$th mode. The gap one prevents relabeling.
Solving \eqref{eq:gas-ray-monodromy} gives
\eqref{eq:gas-side-eigenvalue-asymptotics}.

The normalized eigenvector is $o(1)$ close, up to phase, to the
constant-modulus oscillation $q_tp_\kappa(t\cdot)$.
For fixed $m\ge0$, nonstationary integration gives
$\|Q[\alpha^m q_tp_\kappa(t\cdot)]\|_2=O(t^{-1/2})$;
possible amplitude jumps contribute $O((t+|j|)^{-1})$ coefficients.
Polynomial induction replaces $f(A)$ by multiplication by $f(\alpha)$.
Uniform polynomial approximation and the uniform distribution of
$\alpha$ on $[a,b]$ prove \eqref{eq:gas-side-spectral-marks}.
\end{proof}

\begin{theorem}
\label{gas:half-lattice-gas-ray}
Let $\varrho=\sqrt{a/b}$ and
$\Phi_{\rm ray}(\xi)=\pi^{-1}\operatorname{arctanh}^2((b-a)/(a+b-\xi))$
for $\xi<2a$. For real $\gamma\in C([a,b])$, uniformly on compact
$\xi$-intervals in $(-\infty,2a)$ and bounded $y$-intervals,
\begin{equation}\label{gas:half-lattice-gas-ray-profile}
 \begin{aligned}
 u^\gamma(t,t\xi+y)&=
 K\frac{1-\varrho^2}{1-2\varrho\cos\Theta_\gamma+\varrho^2}+o(1),\\
 \Theta_\gamma&=K(t\xi+y)-K(a+b)t-\int_a^b\gamma(s)\,ds-\Phi_{\rm ray}(\xi),
 \end{aligned}
\end{equation}
as $t\to+\infty$, with every fixed number of $y$-derivatives.
\end{theorem}
\begin{proof}
First let $\gamma=0$. Proposition~\ref{prop:gas-side-eigenvalue-marks}
gives $\lambda_j=dj-2pt+\Delta(j/t)+o(1)$ and
$\langle v_j,Av_j\rangle\to p$ uniformly for $j/t$ in compact
subsets of $(h,\infty)$. In the eigenbasis,
Lemma~\ref{gas:cm-spectral-coordinate-chart} gives off-diagonal
entries $-i/(\lambda_j-\lambda_k)$ and $aI\le A\le bI$.
For any sequence $(t_l,\xi_l)$, extract a subsequence with $\xi_l$
convergent and choose $j_l$ so that
$z_l=t_l\xi_l+dj_l-2pt_l+\Delta((2p-\xi_l)/d)\in[0,d)$.
Then $j_l/t_l=(2p-\xi_l)/d+O(t_l^{-1})$; after $z_l\to z$,
the recentered coordinates converge locally to $z+d\mathbb Z$
and their diagonal matrix elements to $p$.

For $R(y)=(y+\Lambda+iC)^{-1}$, the identity
$R(y)e_j=[e_j-iR(y)(C-I)e_j]/(y+\lambda_j+i)$ and
$\|R\|\le2b$ bound the squared Hilbert--Schmidt tail outside
$[-M,M]$ by $O(M^{-1})$. Square-summable Cauchy tails give strong
convergence of $A$, then of its inverse and the resolvents;
the tail bound gives Hilbert--Schmidt convergence. The quadratic
trace and all its $y$-derivatives therefore converge to the bilateral
reconstruction \eqref{eq:half-cm-bilateral-wave}.
Since $K\Delta((2p-\xi)/d)=-\Phi_{\rm ray}(\xi)$, this is the
case $\gamma=0$ of \eqref{gas:half-lattice-gas-ray-profile}.
The sequential argument proves uniformity.
For general $\gamma$, the translation-invariant, locally smooth
continuous action $\mathcal T_\gamma$
(Theorem~\ref{gas:continuous-norming-action}) acts on each bilateral
limit, shifting its phase by $-\int_a^b\gamma$ as in
\eqref{gas:half-lattice-background}.
\end{proof}

\subsection{Proof of Theorem~\ref{thm:main-gas}: the construction and asymptotics of soliton gases}
\label{subsec:proof-main-gas}
\begin{proof}
Part~\textup{(i)} is Theorem~\ref{gas:arbitrary-band-quasiperiodic-gas}
and Corollary~\ref{gas:arbitrary-band-finite-arrays}.
For \textup{(ii)}, take $\gamma=\gamma_r=\log r/(2s)$ in
Theorem~\ref{gas:half-lattice-gas}. Theorem~\ref{thm:cm-physical-negative-transform}
gives the Lax spectrum; Theorems~\ref{gas:half-lattice-far-fields},
\ref{gas:half-lattice-exterior} and~\ref{gas:half-lattice-gas-ray}
give the spatial and long-time limits.
Theorem~\ref{gas:half-lattice-weight-response}, with $\eta=0$,
gives the exterior response, recovery of $r$, and injectivity.
\end{proof}


\begin{thebibliography}{99}

\bibitem{AhernClark}
{\small \textsc{Ahern, P. R., Clark, D. N.}
Radial limits and invariant subspaces,
{\it American Journal of Mathematics,} {\bf 92} (1970), 332--342.}

\bibitem{gasAltschul19}
{\small \textsc{Altschul, B.}
Statistical equilibrium of the Korteweg--de Vries and Benjamin--Ono unidirectional soliton models,
{\it Journal of Statistical Mechanics: Theory and Experiment,} (2019), 103209.}

\bibitem{BKV26}
{\small \textsc{Badreddine, R., Killip, R., Vi\c{s}an, M.}
Orbital stability of Benjamin--Ono multisolitons,
{\it arXiv preprint arXiv:2509.14153v2,} (2026).}

\bibitem{BKVasymp}
{\small R. Badreddine, R. Killip, and M. Vi\c{s}an,
Asymptotic stability of Benjamin--Ono multisolitons in $L^2(\mathbb R)$,
arXiv:2606.12951 (2026).}

\bibitem{Benjamin}
{\small \textsc{Benjamin, T. B.}
Internal waves of permanent form in fluids of great depth,
{\it Journal of Fluid Mechanics,} {\bf 29} (1967), 559--592.}

\bibitem{BGGM}
{\small E. Blackstone, L. Gassot, P. G\'erard, and P. D. Miller,
The Benjamin--Ono initial-value problem for rational data with application
to long-time asymptotics and scattering,
\emph{Ann. Inst. H. Poincar\'e C Anal. Non Lin\'eaire} (2025),
published online, doi:10.4171/AIHPC/169.}

\bibitem{BJM18} {\small \textsc{M. Borghese, R. Jenkins, K. D. T.-R. McLaughlin} Long time asymptotic behavior of the focusing nonlinear Schr\"odinger equation, {\it Annales de l'Institut Henri Poincar\'e C, Analyse non lin\'eaire,} {\bf 35} (2018), 887--920.}

\bibitem{Chen}
{\small X. Chen, Explicit formula for the Benjamin--Ono equation with square integrable
and real valued initial data and applications to the zero dispersion limit,
\emph{Pure Appl. Anal.} \textbf{7} (2025), 101--126.}

\bibitem{DZ93} {\small \textsc{P. A. Deift, X. Zhou} A steepest descent method for oscillatory Riemann--Hilbert problems. Asymptotics for the MKdV equation, {\it Annals of Mathematics,} {\bf 137} (1993), 295--368.}

\bibitem{gasDK91}
{\small \textsc{Dobrokhotov, S. Yu., Krichever, I. M.}
Multi-phase solutions of the Benjamin--Ono equation and their averaging,
\emph{Mathematical Notes} \textbf{49} (1991), 583--594.}

\bibitem{gasEl03}
{\small \textsc{El, G. A.}
The thermodynamic limit of the Whitham equations,
{\it Physics Letters A,} {\bf 311} (2003), 374--383.}

\bibitem{gasEK05}
{\small \textsc{El, G. A., Kamchatnov, A. M.}
Kinetic equation for a dense soliton gas,
{\it Physical Review Letters,} {\bf 95} (2005), 204101.}

\bibitem{gasET20}
{\small \textsc{El, G., Tovbis, A.}
Spectral theory of soliton and breather gases for the focusing nonlinear Schr\"odinger equation,
{\it Physical Review E,} {\bf 101} (2020), 052207.}

\bibitem{GarciaModel}
{\small \textsc{Garcia, S. R., Mashreghi, J., Ross, W. T.}
{\it Introduction to Model Spaces and Their Operators}, Cambridge University Press, Cambridge, (2016).}

\bibitem{Garnett}
{\small \textsc{Garnett, J. B.}
{\it Bounded Analytic Functions}, revised first edition, Springer, New York, (2007).}

\bibitem{GG}
{\small \textsc{Gassot, L., G\'erard, P.}
Infinite-order multisoliton solutions to the Benjamin--Ono equation and soliton resolution,
{\it arXiv preprint arXiv:2603.15419v2,} (2026).}

\bibitem{GGM}
{\small \textsc{Gassot, L., G\'erard, P., Miller, P. D.}
A proof of the soliton resolution conjecture for the Benjamin--Ono equation,
{\it arXiv preprint arXiv:2601.10488v2,} (2026).}

\bibitem{GGMExplicit26}
{\small \textsc{Gassot, L., G\'erard, P., Miller, P. D.}
An explicit formula for the solution of the Benjamin--Ono equation and its applications,
{\it arXiv preprint arXiv:2610.00770v1,} (2026).}

\bibitem{Gerard}
{\small P. G\'erard, An explicit formula for the Benjamin--Ono equation,
\emph{Tunis. J. Math.} \textbf{5} (2023), 593--603.}

\bibitem{GKT}
{\small P. G\'erard, T. Kappeler, and P. Topalov,
Sharp well-posedness results of the Benjamin--Ono equation in
$H^s(\mathbb T,\mathbb R)$ and explicit properties of its solutions,
\emph{Acta Math.} \textbf{231} (2023), 31--88.}

\bibitem{GL24} {\small \textsc{P. G\'erard, E. Lenzmann} The Calogero--Moser derivative nonlinear Schr\"odinger equation, {\it Communications on Pure and Applied Mathematics,} {\bf 77} (2024), 4008--4062.}

\bibitem{gasGGJM21}
{\small \textsc{Girotti, M., Grava, T., Jenkins, R., McLaughlin, K. D. T.-R.}
Rigorous asymptotics of a KdV soliton gas,
{\it Communications in Mathematical Physics,} {\bf 384} (2021), 733--784.}

\bibitem{gasGGJMM23}
{\small \textsc{Girotti, M., Grava, T., Jenkins, R., McLaughlin, K. D. T.-R., Minakov, A.}
Soliton versus the gas: Fredholm determinants, analysis, and the rapid oscillations behind the kinetic equation,
{\it Communications on Pure and Applied Mathematics,} {\bf 76} (2023), 3233--3299.}

\bibitem{IfrimTataru} {\small \textsc{M. Ifrim, D. Tataru} Well-posedness and dispersive decay of small data solutions for the Benjamin--Ono equation, {\it Annales Scientifiques de l'\'Ecole Normale Sup\'erieure,} {\bf 52} (2019), 297--335.}

\bibitem{IonescuKenig}
{\small \textsc{Ionescu, A. D., Kenig, C. E.}
Global well-posedness of the Benjamin--Ono equation in low-regularity spaces,
{\it Journal of the American Mathematical Society,} {\bf 20} (2007), 753--798.}

\bibitem{KenigMartel} {\small \textsc{C. E. Kenig, Y. Martel} Asymptotic stability of solitons for the Benjamin--Ono equation, {\it Revista Matem\'atica Iberoamericana,} {\bf 25} (2009), 909--970.}

\bibitem{KLV24}
{\small \textsc{Killip, R., Laurens, T., Vi\c{s}an, M.}
Sharp well-posedness for the Benjamin--Ono equation,
{\it Inventiones mathematicae,} {\bf 236} (2024), no.~3, 999--1054.}

\bibitem{KK24}
{\small \textsc{Kim, T., Kwon, S.}
Soliton resolution for Calogero--Moser derivative nonlinear Schr\"odinger equation,
{\it arXiv preprint arXiv:2408.12843v2,} (2026), to appear in
{\it Journal of the European Mathematical Society.}}

\bibitem{KS21}
{\small \textsc{Klein, C., Saut, J.-C.}
\emph{Nonlinear Dispersive Equations: Inverse Scattering and PDE Methods},
Applied Mathematical Sciences, vol.~209, Springer, Cham, 2021.}

\bibitem{Matsuno79}
{\small \textsc{Matsuno, Y.}
Exact multi-soliton solution of the Benjamin--Ono equation,
{\it Journal of Physics A: Mathematical and General,} {\bf 12} (1979), no.~4, 619--621.}

\bibitem{MP12} {\small \textsc{L. Molinet, D. Pilod} The Cauchy problem for the Benjamin--Ono equation in $L^2$ revisited, {\it Analysis \& PDE,} {\bf 5} (2012), 365--395.}

\bibitem{Ono75}
{\small \textsc{Ono, H.}
Algebraic solitary waves in stratified fluids,
{\it Journal of the Physical Society of Japan,} {\bf 39} (1975), no.~4, 1082--1091.}

\bibitem{Saut}
{\small J.-C. Saut, Sur quelques g\'en\'eralisations de l'\'equation de Korteweg--de Vries,
\emph{J. Math. Pures Appl.} \textbf{58} (1979), 21--61.}

\bibitem{Saut2}
{\small \textsc{Saut, J.-C.}
Benjamin--Ono and intermediate long wave equations: modeling, IST and PDE,
in \emph{Nonlinear Dispersive Partial Differential Equations and Inverse Scattering},
Fields Institute Communications, vol.~83, Springer, New York, 2019, pp.~95--160.}

\bibitem{Sun21}
{\small \textsc{Sun, R.}
Complete integrability of the Benjamin--Ono equation on the multi-soliton manifolds,
{\it Communications in Mathematical Physics,} {\bf 383} (2021), 1051--1092.}

\bibitem{TaoBO}
{\small T. Tao, Global well-posedness of the Benjamin--Ono equation in $H^1(\mathbb R)$,
\emph{J. Hyperbolic Differ. Equ.} \textbf{1} (2004), 27--49.}

\bibitem{TaoSolitons}
{\small \textsc{Tao, T.}
Why are solitons stable?,
\emph{Bulletin of the American Mathematical Society}
\textbf{46} (2009), no.~1, 1--33.}

\bibitem{Wu}
{\small Y. Wu, Simplicity and finiteness of discrete spectrum of the Benjamin--Ono scattering operator,
\emph{SIAM J. Math. Anal.} \textbf{48} (2016), 1348--1367.}

\bibitem{WuDirect}
{\small Y. Wu, Jost solutions and the direct scattering problem of the Benjamin--Ono equation,
\emph{SIAM J. Math. Anal.} \textbf{49} (2017), 5158--5206.}

\bibitem{ZabuskyKruskal} {\small \textsc{N. J. Zabusky, M. D. Kruskal} Interaction of solitons in a collisionless plasma and the recurrence of initial states, {\it Physical Review Letters,} {\bf 15} (1965), 240--243.}

\bibitem{gasZak71}
{\small \textsc{Zakharov, V. E.}
Kinetic equation for solitons,
{\it Soviet Physics JETP,} {\bf 33} (1971), 538--541.}

\end{thebibliography}
\end{document}